\documentclass[11pt,oneside,reqno]{amsart}

\usepackage{amsmath}
\usepackage{amsfonts}
\usepackage[lite]{amsrefs}
\usepackage{amssymb}
\usepackage{amsthm}                
\usepackage{comment}
\usepackage{enumerate}
\usepackage{epsfig}
\usepackage{epstopdf}
\usepackage[mathscr]{euscript}
\usepackage{float}
\usepackage[margin=1in]{geometry}  
\usepackage{graphics}
\usepackage{graphicx}
\usepackage{hyperref}
\usepackage[utf8]{inputenc}
\usepackage{latexsym}
\usepackage{longtable}
\usepackage{mathdots}
\usepackage{pict2e}
\usepackage{subfigure}
\usepackage{tikz}

\renewcommand{\PrintDOI}[1]{\href{http://dx.doi.org/\detokenize{#1}}{doi: \detokenize{#1}}%
	\IfEmptyBibField{pages}{, (to appear in print)}{}}

\theoremstyle{definition}
\newtheorem{theorem}{Theorem}[section]
\newtheorem{lemma}[theorem]{Lemma}
\newtheorem{corollary}[theorem]{Corollary}
\newtheorem{proposition}[theorem]{Proposition}

\theoremstyle{definition}
\newtheorem{definition}[theorem]{Definition}
\newtheorem{example}[theorem]{Example}
\theoremstyle{remark}
\newtheorem{remark}[theorem]{Remark}
\numberwithin{equation}{section}

\numberwithin{equation}{section}

\begin{document}

\title{$Q$-Series and Quantum Spin Networks}

\author{Mohamed Elhamdadi}
\address{Department of Mathematics, University of South Florida, 
Tampa, FL USA}
\email{emohamed@mail.usf.edu}

\author{Mustafa Hajij}
\address{Department of Computer Science and Engineering, Ohio State University, 
Columbus, Ohio USA}
\email{hajij.1@osu.edu}

\author{Jesse S F Levitt}
\address{Department of Mathematics, University Of Southern California, 
Los Angeles, CA USA}
\email{jslevitt@usc.edu}





\begin{abstract}
The tail of a quantum spin network in the two-sphere is a $q$-series associated to the network.
We study the existence of the head and tail functions of quantum spin networks colored by $2n$.
We compute the $q$-series for an infinite family of quantum spin networks and give the relation between the tail of these networks and the tail of the colored Jones polynomial.
Finally, we show that the family of quantum spin networks under study satisfies a natural product structure, making these networks satisfy a natural product structure.
\end{abstract}

\maketitle


\section{Introduction}

The colored Jones polynomial assigns to every link $L$ a sequence of Laurent polynomials $\left\{J_{n,L}\right\}_{n\in \mathbb{N}}$ where the positive integer $n$ is called the \textit{color}, see~\cite{TuraevWenzl93}.
Recent advances in the study of this polynomial showed that for alternating and adequate knots, certain coefficients of $J_{n,L}$ stabilize as $n$ increases~\cites{Armond1, EH, GL, Hajij2,BEH, EHS, EH2}. 
More precisely, for any alternating link $L$ the first $(n+1)$-coefficients of $J_{n,L}$ agree with the 
initial $(n+1)$-coefficients of $J_{n+1,L}$.
This gives rise to a $q$-series called the \textit{tail of the colored Jones polynomial}.
The highest degree coefficients of the colored Jones polynomial have similar stability properties and this induced power series is instead called the \textit{head} of the colored Jones polynomial.
This behavior was first observed by Dasbach and Lin~\cite{DL} and was proved by Armond~\cite{Armond1} for adequate links and independently by Garoufalidis and L\^e~\cite{GL} who also showed higher order coefficient stability.
One of the interesting aspects of the $q$-series coming from the colored Jones polynomial is their relation to the Ramanujan theta and false theta functions.
Armond and Dasbach~\cite{CodyOliver} used the properties of the colored Jones polynomial to prove the Andrew-Gordan identity for theta functions~\cite{AB}.
A corresponding identity for the false theta functions was given by the second author~\cite{Hajij2}.
The stability of the coefficients of other quantum invariants have also been studied recently.
For instance, in~\cite{Wataru} the coefficients of the $\mathfrak{sl}_{3}(\mathbb {C})$-colored Jones polynomial were used to give a generalization for the identity given by the second author in~\cite{Hajij1}.

Let $D$ be a planar trivalent graph in the $2$-sphere $S^2$.
Fix a positive integer $n$ and label every edge in $D$ by $n$ or $2n$ such that we obtain an admissible quantum spin network $D_n$ (see the precise definition in Section~\ref{background}).
This defines a sequence of quantum spin networks $\mathcal{D}=\{D_n\}_{n \in \mathbb{N}}$.
In~\cite{Hajij1} the second author initiated a study of the stability of the coefficients of the evaluations of the sequence elements of $\mathcal{ D}$, a study which arises naturally when one considers the tail of colored Jones polynomial~\cite{Armond1}.
Previous work~\cite{Hajij2} has shown that the quantum spin networks (QSNs) corresponding to adequate skein elements admit a well-defined tail.
However, it was also found that the tail might exist for QSNs whose skein elements are not adequate. 
  
In this paper we focus on quantum spin networks with all edges colored $2n$.
We show that the tail of such networks always exist and show how to compute the tail of such networks on an infinite family of graphs.
Additionally, these networks satisfy a natural product structure.
We further illustrate the relationship of the tail of these graphs to the tail of the colored Jones polynomial of alternating links.
Finally, we demonstrate how the tail of an infinite family of alternating links can be computed by considering the tail of a single QSN.
  
The paper is organized as follows.
In section~\ref{background}, we recall the necessary background needed for the paper. 
Section~\ref{QSNet} defines the admissibility of a QSN and discusses the Kauffman bracket evaluation.
In section~\ref{TQSN}, we recall the definition of the tail of a QSN and show that the tail of any QSN whose edges are all labeled by $2n$ exists.
Section~\ref{Application} covers a connection of the tail of a QSN to the tail of related colored Jones polynomial.
While section ~\ref{Prod} deals with the product structure on the tail of two QSNs.
As in the case of the colored Jones polynomial, the tail of quantum spin networks with edges colored $2n$ satisfies a natural product structure.
In section~\ref{Computing} we give the tail of the theta and tetrahedron graphs with edges colored $2n$.
We then use this and the theta and tetrahedron graphs to compute the tail of infinite families of other graphs.


\section{Background}
\label{background}

Let $F$ be a connected oriented surface, with boundary denoted $\partial F$.
When the boundary $\partial F$ is non-empty and a finite set of marked points are chosen on it, a link diagram in $F$ is a finite collection of arcs and simple closed curves in $F$ that meet $\partial F$ orthogonally at the marked points.
As in the case of standard link diagrams, the link diagram in $F$ will be assumed to have a finite number of crossing points.
Moreover, at crossings we will distinguish the strands using the usual convention of upper-strand and lower-strand.
We will work over $\mathcal{R}=\mathbb{Q}(A)$, the field generated by the indeterminate $A$ over the rational numbers. Furthermore set $A^4=q.$

\begin{definition}
\label{SkeinDef}
    Let $\mathcal{D}(F)$ be the free $\mathcal{R}$-module of link diagrams in $F$.
    The \textit{linear skein} $\mathcal{S}(F)$ of $F$ is the quotient of the module $\mathcal{D}(F)$ by the relations:
    \begin{eqnarray*}(1)\hspace{3 mm}
 	    \begin{minipage}[h]{0.06\linewidth}
 		    \vspace{0pt}
 	    	\scalebox{0.04}{\includegraphics{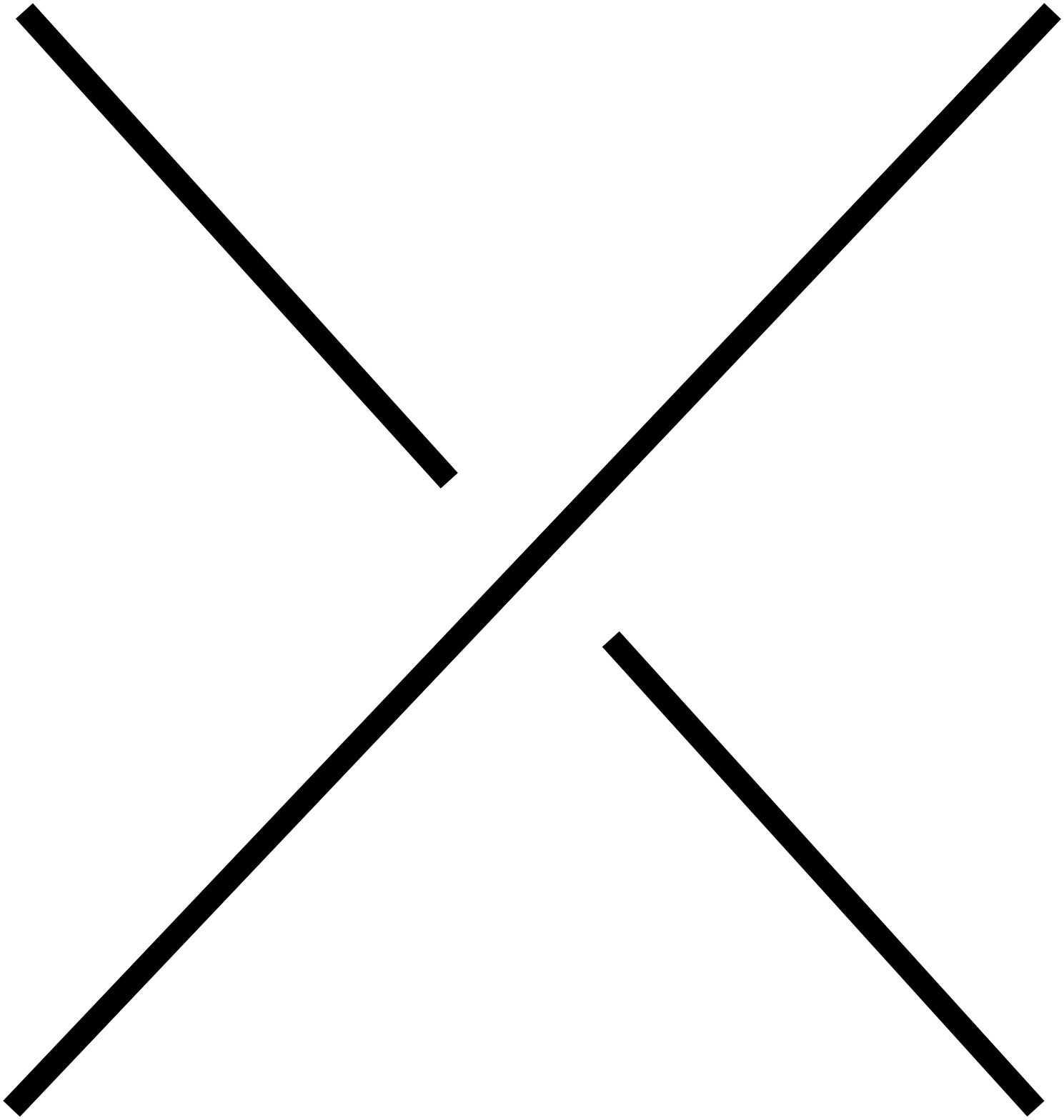}}
    	\end{minipage}
 	    -A
     	\begin{minipage}[h]{0.06\linewidth}
 		    \vspace{0pt}
 	    	\scalebox{0.04}{\includegraphics{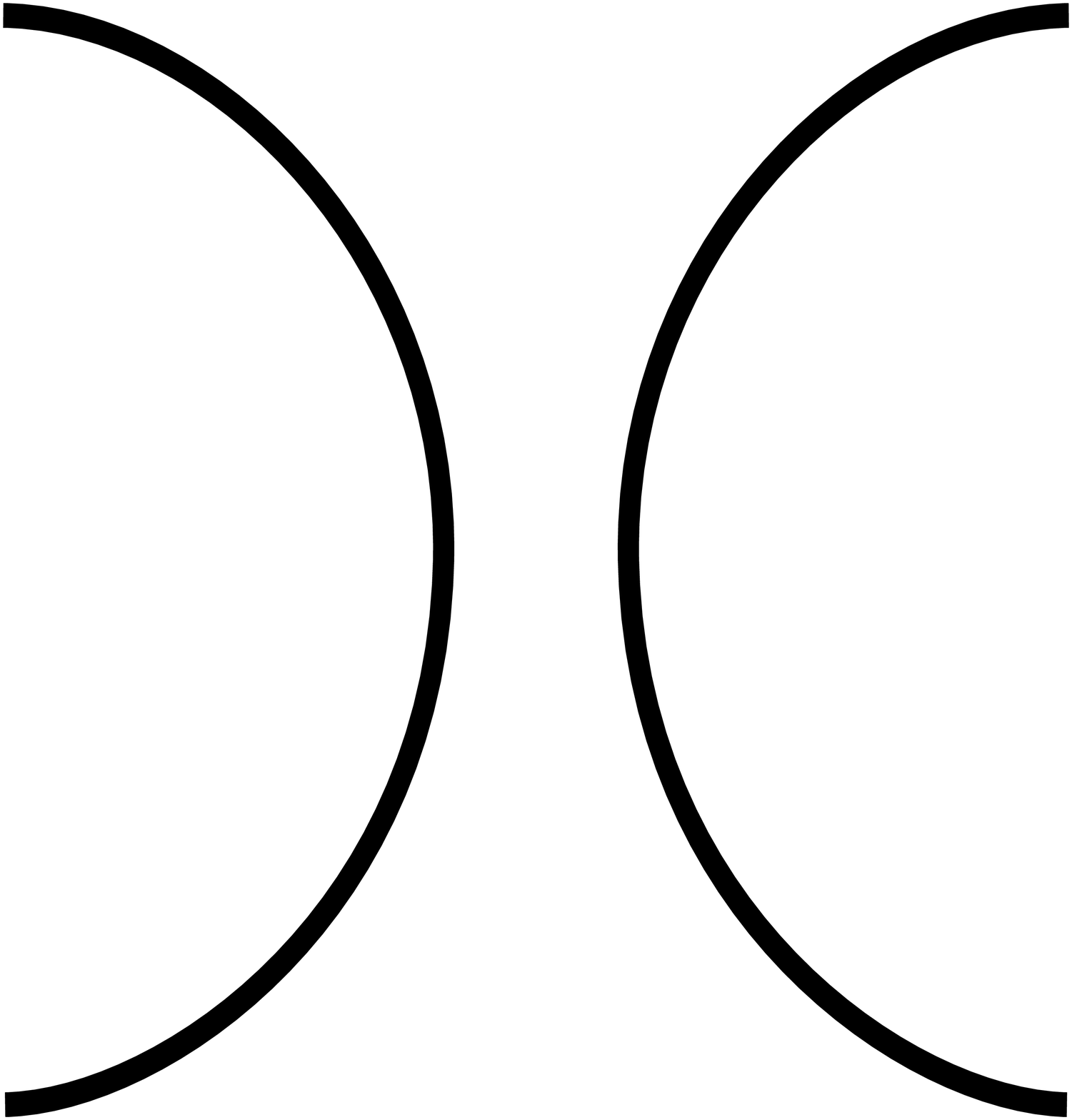}}
    	\end{minipage}
 	    -A^{-1} 
    	\begin{minipage}[h]{0.06\linewidth}
 		    \vspace{0pt}
 		    \scalebox{0.04}{\includegraphics{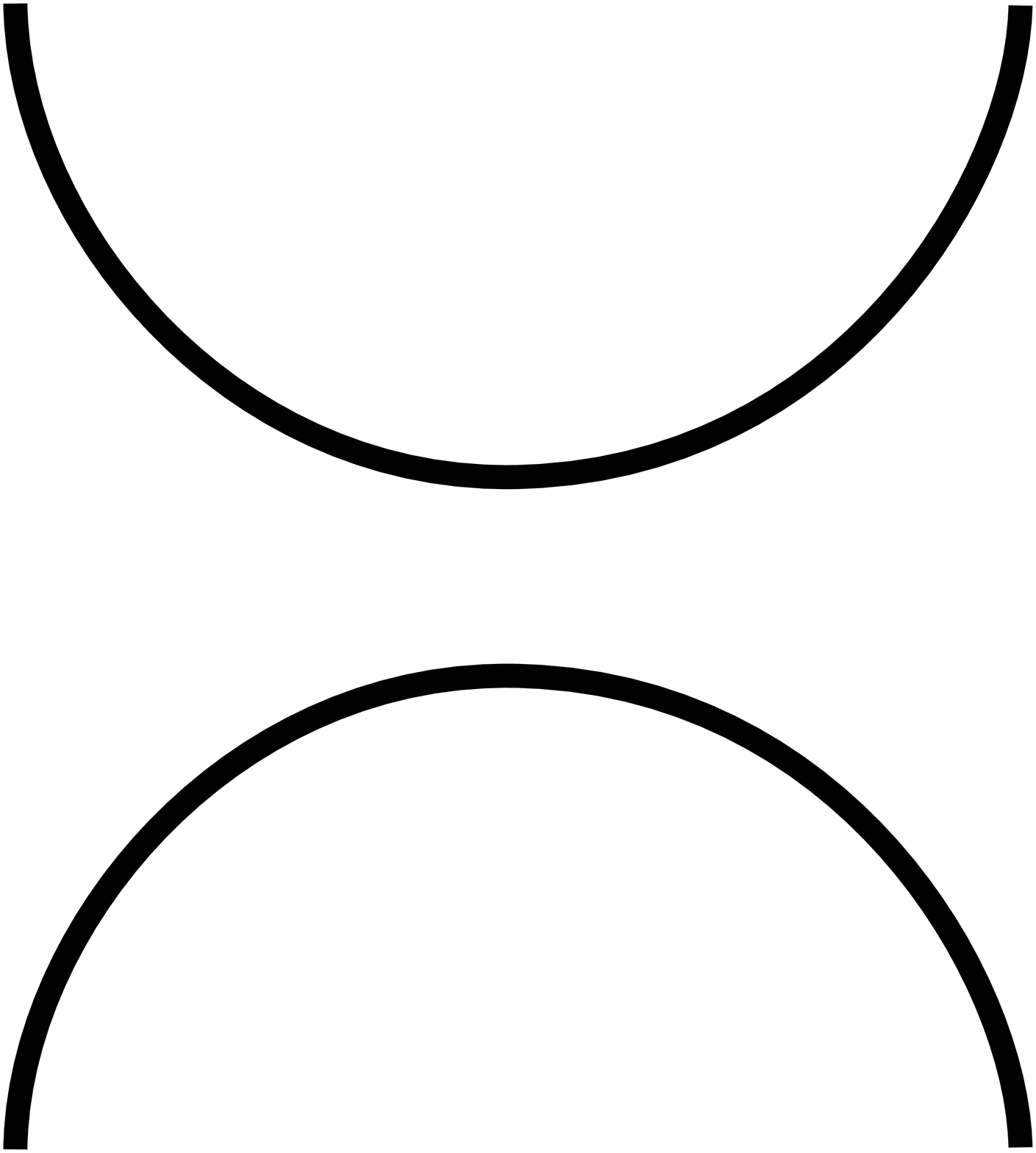}}
	    \end{minipage}
    	, \hspace{20 mm}
        (2)\hspace{3 mm} L\sqcup\hspace{0.4mm}
    	\begin{minipage}[h]{0.05\linewidth}
 		    \vspace{0pt}
 	    	\scalebox{0.02}{\includegraphics{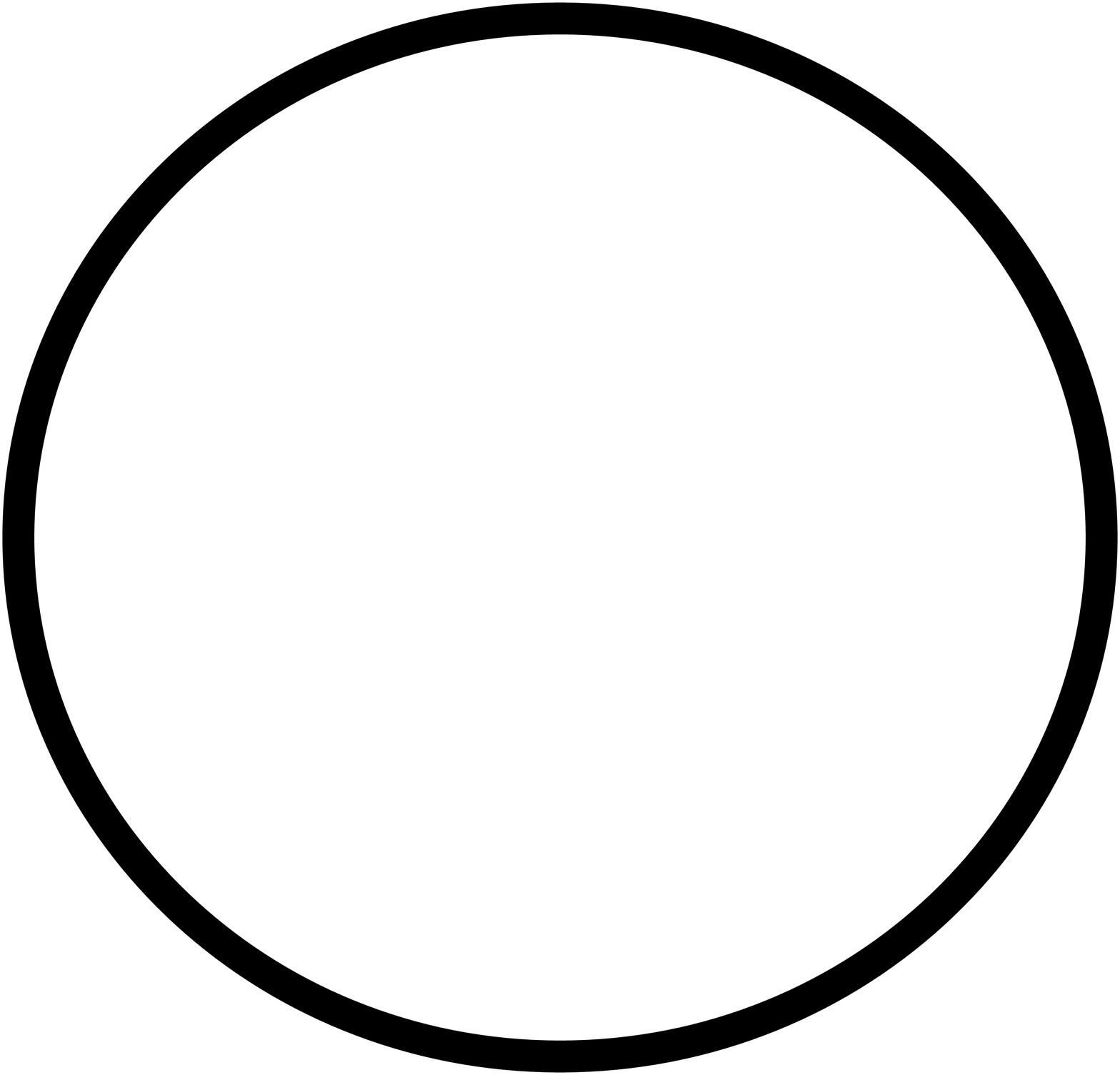}}
    	\end{minipage}
 	    +(A^2+A^{-2})L. 
    \end{eqnarray*}
    where $L\,\sqcup$
    \begin{minipage}[h]{0.05\linewidth}
 	    \vspace{0pt}
 	    \scalebox{0.02}{\includegraphics{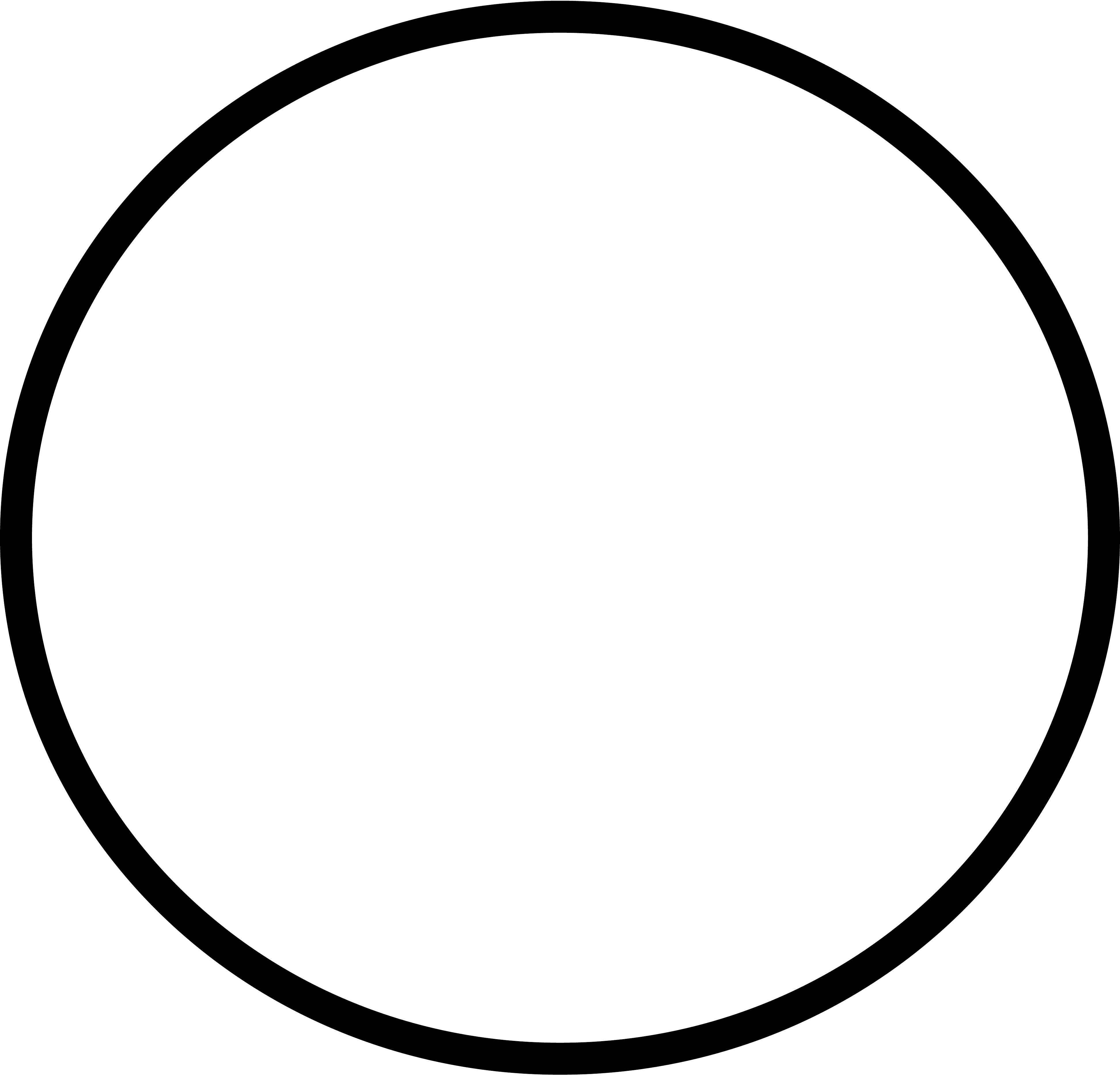}}
    \end{minipage}
    consists of a link $L$ in $F$ and a disjoint simple connected curve  
    \begin{minipage}[h]{0.05\linewidth}
        \vspace{0pt}
 	    \scalebox{0.02}{\includegraphics{simple-circle}}
    \end{minipage} that is null-homotopic in $F$.
\end{definition}    
The linear skein space is also called \textit{the Kauffman bracket skein module}~\cites{Przytycki, RT}.
The two main linear skein spaces needed in this paper are the linear skein of the 2-sphere and the linear skein of the disk with some marked points on the boundary.
The linear skein module of the sphere $S^2$ is isomorphic to the ring $\mathbb{Q}(A)$.
To describe the linear skein space of the disk with boundary and marked points, first let $E=I\times I$ where $I=[0,1]$ and then fix $2n$ marked points on the boundary of $E$, with precisely $n$ points on the top and $n$ points on the bottom of $E$.
We then denote by $\mathcal{S}(E,2n)$ the linear skein module of the disk $E$ with $2n$ marked points.
We make this into an associative algebra over $\mathbb{Q}(A)$ by the natural vertical juxtaposition of diagrams known as the \textit{$n^{th}$ Temperley-Lieb algebra} $TL_n$.
The \textit{Jones-Wenzl idempotent} (projector), denoted $f^{(n)}$, is an idempotent in $TL_n$.
The graphical depiction for this projector appears as a box labelled with one strand entering the box from one side and one strand leaving the box from the other side of the box.
The label $n$ is usually drawn next to the box to indicate label of the projector. 

The projector can be characterized completely by the first two axioms in~\ref{propertiesJ}, with the latter two relations following as a consequence.
For more information about these important idempotents, including a useful recursive relation, we recommend Wenzl's critical paper~\cite{Wenzl}.
\begin{eqnarray}
\label{propertiesJ}
    \hspace{-15 mm}
    \begin{minipage}[h]{0.21\linewidth}
        \vspace{0pt}
        \scalebox{0.115}{\includegraphics{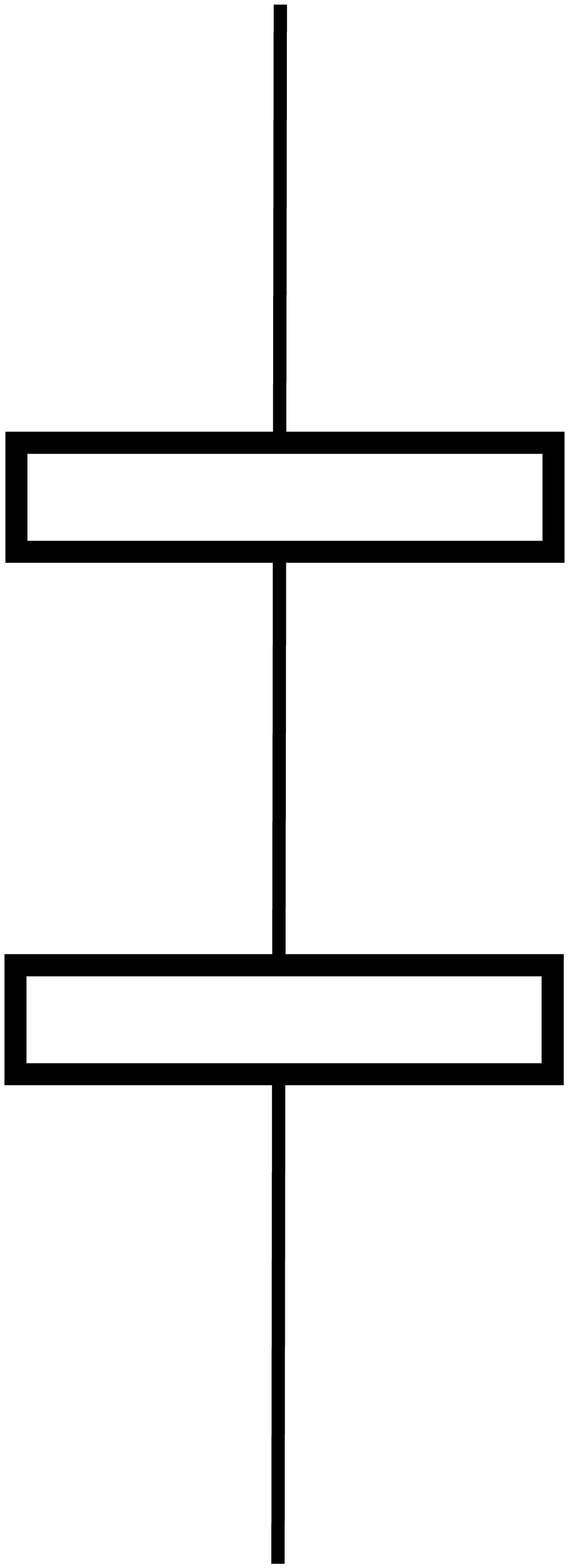}}
        \put(0,+80){\footnotesize{$n$}}
    \end{minipage}
     = \hspace{5pt}
    \begin{minipage}[h]{0.1\linewidth}
        \vspace{0pt}
        \hspace{100pt}
        \scalebox{0.115}{\includegraphics{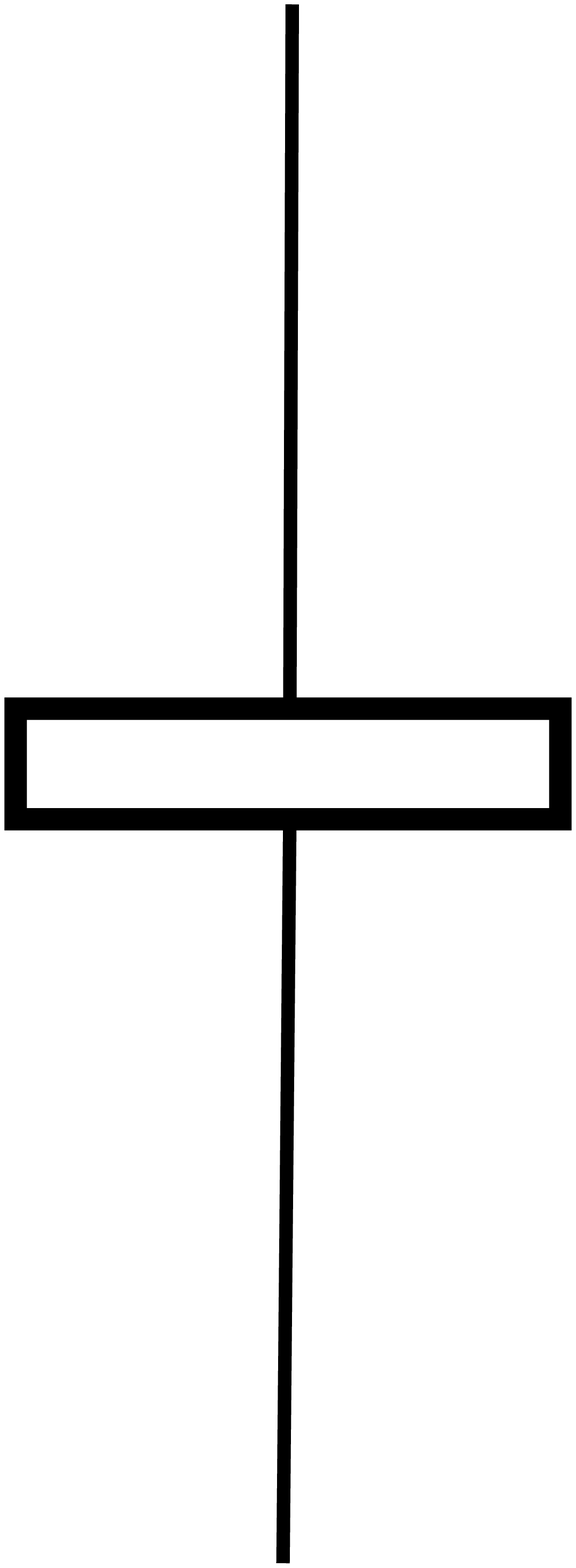}}
        \put(-60,80){\footnotesize{$n$}}
    \end{minipage}
    , \hspace{13 mm}
    \begin{minipage}[h]{0.09\linewidth}
        \vspace{0pt}
        \scalebox{0.115}{\includegraphics{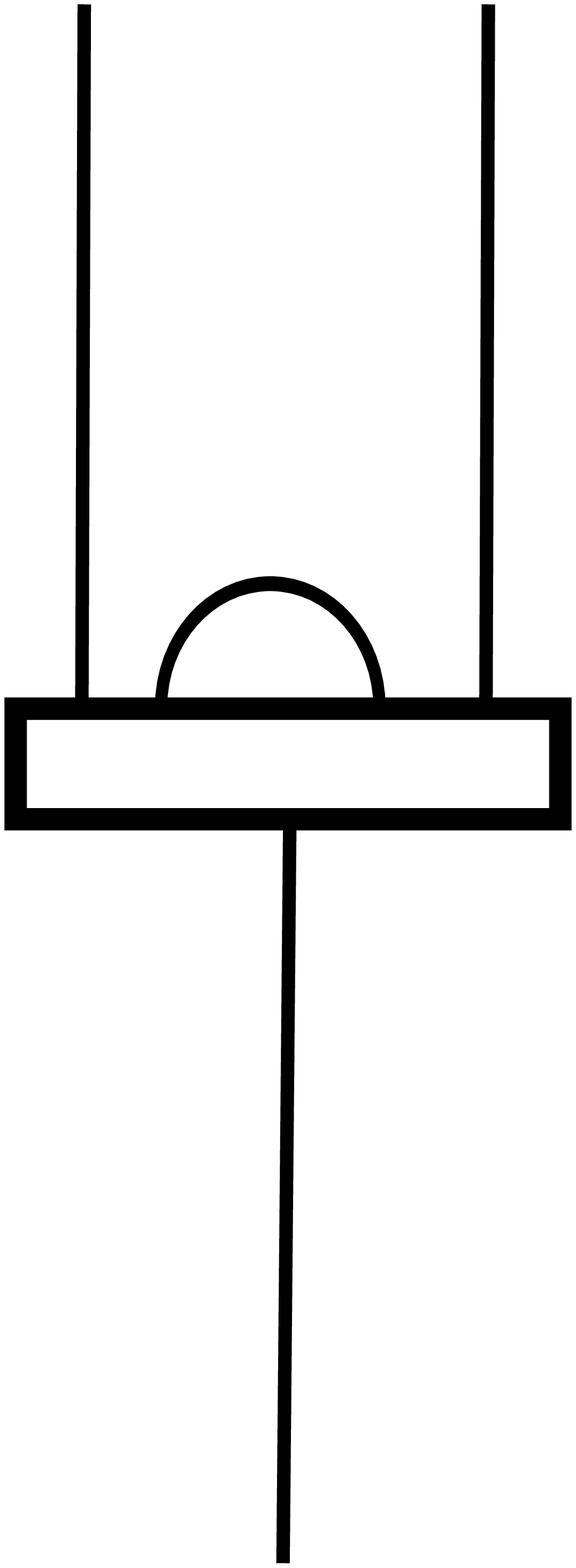}}
        \put(-70,+82){\footnotesize{$n-i-2$}}
        \put(-20,+64){\footnotesize{$1$}}
        \put(-2,+82){\footnotesize{$i$}}
        \put(-28,20){\footnotesize{$n$}}
    \end{minipage}
    =0\;\; \mathrm{and}
    \label{AX}
    \hspace{6 pt}
    \label{propertiesW}
    \hspace{0 mm}
    \begin{minipage}[h]{0.1\linewidth}
        \vspace{0pt}
        \scalebox{0.12}{\includegraphics{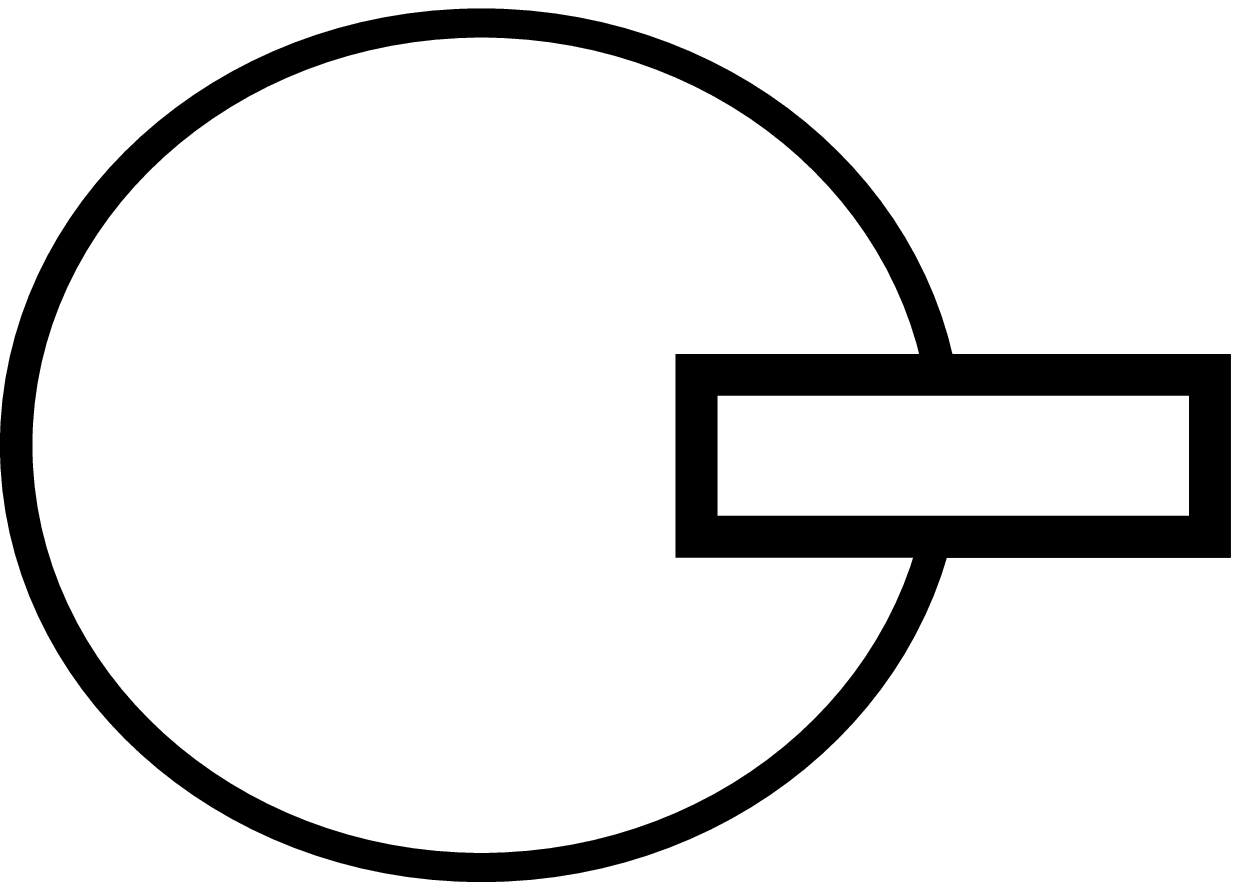}}
        \put(-29,+34){\footnotesize{$n$}}
    \end{minipage}=\Delta_{n}
    , \hspace{12 pt}
    \begin{minipage}[h]{0.08\linewidth}
        \vspace{0pt}
        \scalebox{0.115}{\includegraphics{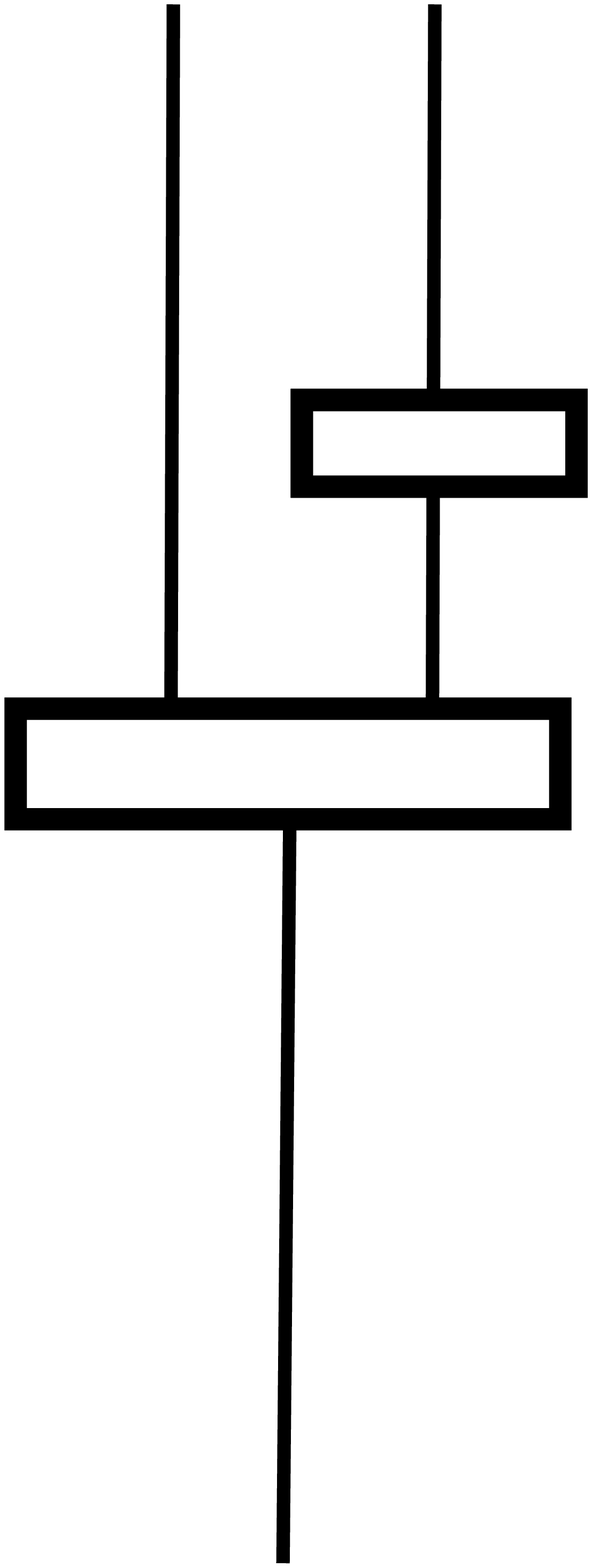}}
        \put(-34,+82){\footnotesize{$n$}}
        \put(-19,+82){\footnotesize{$m$}}
        \put(-46,20){\footnotesize{$m+n$}}
    \end{minipage}
    =\hspace{5pt}
    \begin{minipage}[h]{0.1\linewidth}
        \hspace{10pt}
        \scalebox{0.115}{\includegraphics{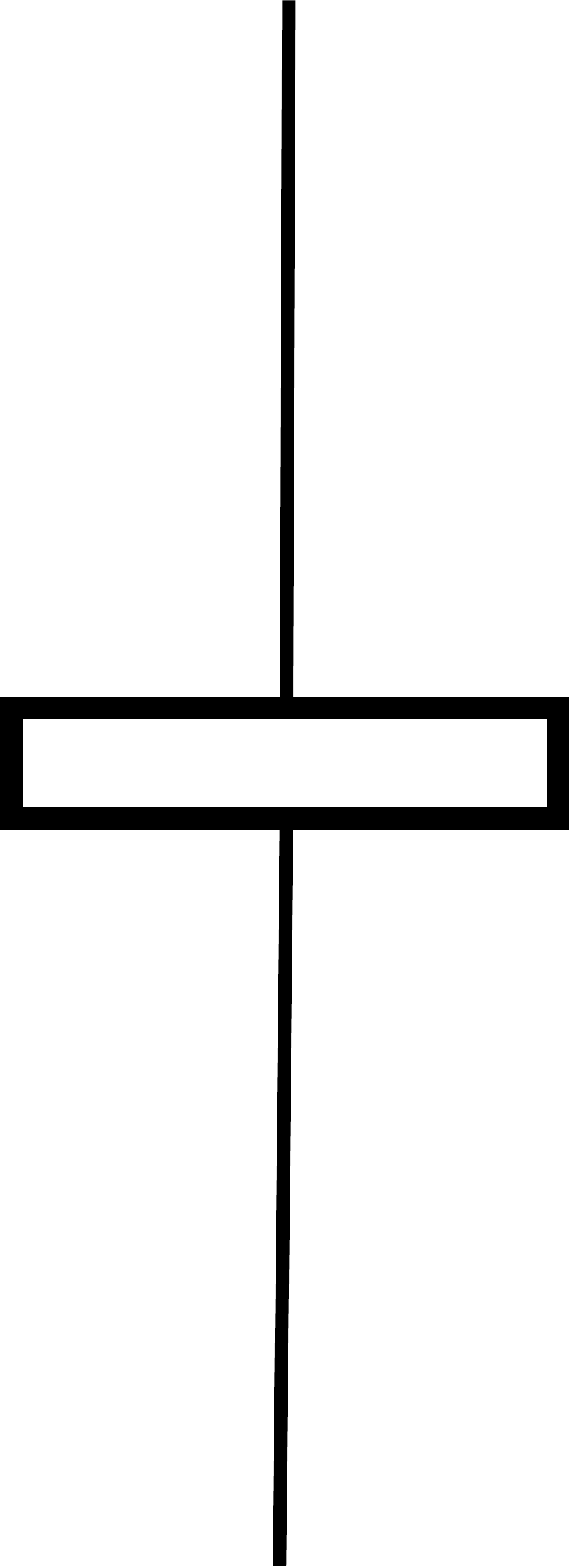}}
        \put(-9,82){\footnotesize{$m+n$}}
    \end{minipage}
\end{eqnarray}
where
\begin{equation*}
    \Delta_{n}=(-1)^{n}\left(\frac{q^{(n+1)/2}-q^{-(n+1)/2}}{q^{1/2}-q^{-1/2}}\right).
\end{equation*}
and the element $\Delta_{n}$ is related to the $(n+1)^{th}$ quantum integer, denoted by $[n+1]_q$, via the equation $\Delta_{n}=(-1)^n [n+1]_q$.

We use this idempotent $f^{(n)}$ to define useful submodules of the Kauffman bracket skein module of the disk with marked points on its boundary as follows.
Let $E$ be a disk with $m$ marked points on its boundary.
Partitioning this set of points into $k$ clusters of $s_i$ ($1\leq i \leq k$) marked points each, with $m=s_1+s_2+\ldots+s_k$, we consider the skein submodule of the skein module of the disk with $m$ marked points obtained by placing idempotents $f^{(s_i)}$ on each cluster of $s_i$ points following clockwise around the disk.
The submodule is thus obtained by taking any diagram $D$ in $\mathcal{S}(E,m)$ and mapping it into the same diagram with the idempotents, $f^{(s_i)}$, placed on the clusters of marked points $s_i$.
We will denote this skein module by $Y_{s_1,\ldots,s_k}$.
Figure~\ref{example of new element} illustrates an example of elements in the skein module of $\mathcal{S}(E,12)$ mapping to elements in the submodule $Y_{4,3,4,1}$.

\begin{figure}[h]
    \centering
    {\includegraphics[scale=0.15]{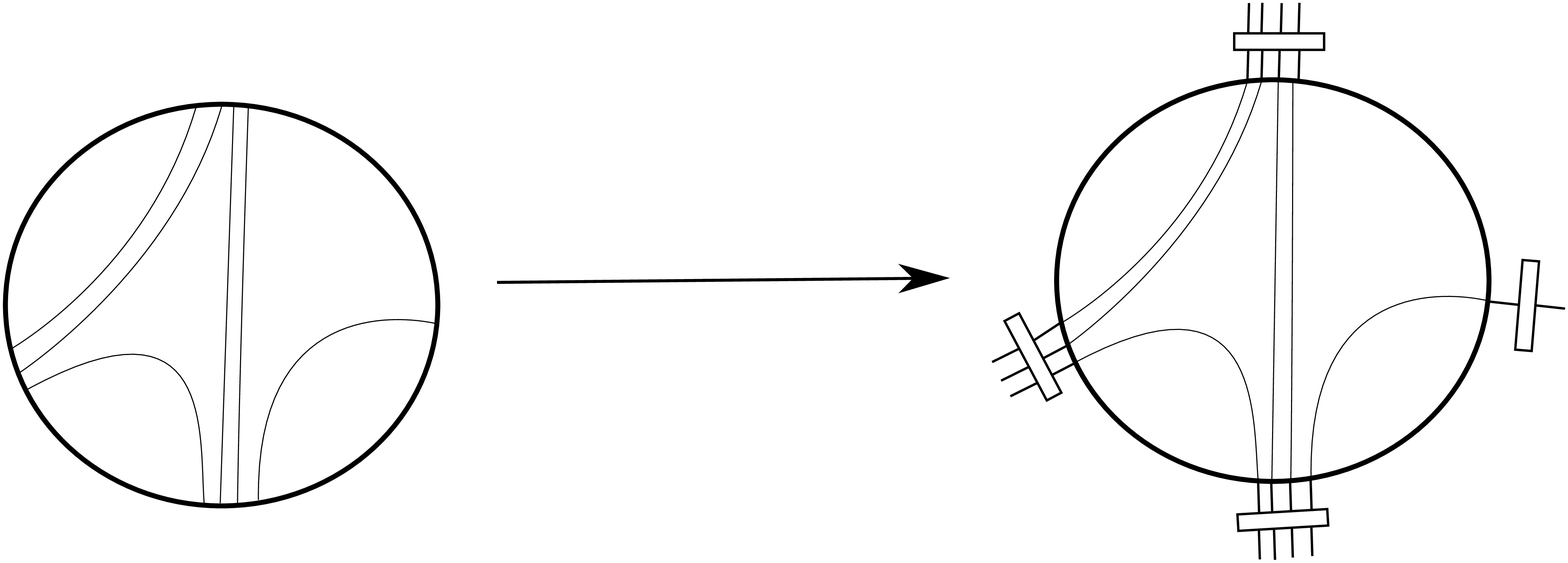}
        \caption{ An element in the skein module of the disk with $4+3+4+1$ marked points on the boundary and the corresponding element in the space $Y_{4,3,4,1}.$ }
    \label{example of new element}}
\end{figure}

The space $Y_a$ is zero dimensional as is $Y_{a,b}$ when $a \neq b$, while $Y_{a,b}$ is one dimensional and generated by $f^{(a)}$ when $a=b$. This follows from the basic properties of the idempotent in~\ref{propertiesJ}. Similarly, the space $Y_{a,b,c}$ is either zero dimensional or one dimensional. It is one dimensional when $a+b+c$ is even and $a+b \geq c \geq |a-b|$. Such a triple $(a,b,c)$ is called \textit{admissible}. When $(a,b,c)$ is admissible the space $Y_{a,b,c}$ is generated by the skein element $\tau_{a,b,c}$ in Figure $\ref{taw1}$. This element exists if and only if the following three equations are satisfied:
\begin{equation}
\label{threeequations}
    a=x+y,\hspace{6pt} b=x+z,\hspace{6pt} c=y+z\quad\mathrm{for}\; x,y,z\in\mathbb{N}.
\end{equation}
\begin{figure}[h]
    \centering
    {\includegraphics[scale=0.25]{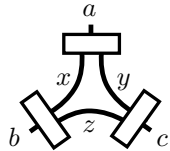}
        \small{
            \put(-58,2){$b$}
            \put(-30,51){$a$}
            \put(-40,25){$x$}
            \put(-17,25){$y$}
            \put(-30,7){$z$}
            \put(-2,2){$c$}}
        \caption{The skein element $\tau_{a,b,c}$ in the space $Y_{a,b,c\text{ }}$ }
    \label{taw1}}
\end{figure}


\section{Quantum Spin Networks}
\label{QSNet}

A \textit{quantum spin network} (QSN) is a planar trivalent graph with edges labeled by non-negative integers.
A zero-labeled edge corresponds to deleting that edge.
We say that a QSN is \textit{admissible} if the three labels at every vertex satisfy the admissibility conditions~\ref{threeequations}, otherwise it is \textit{inadmissible}. 
See Figure~\ref{QSN} for an example and non-example of an admissible quantum spin network.
\begin{figure}[h]
    \centering
    {\includegraphics[scale=0.15]{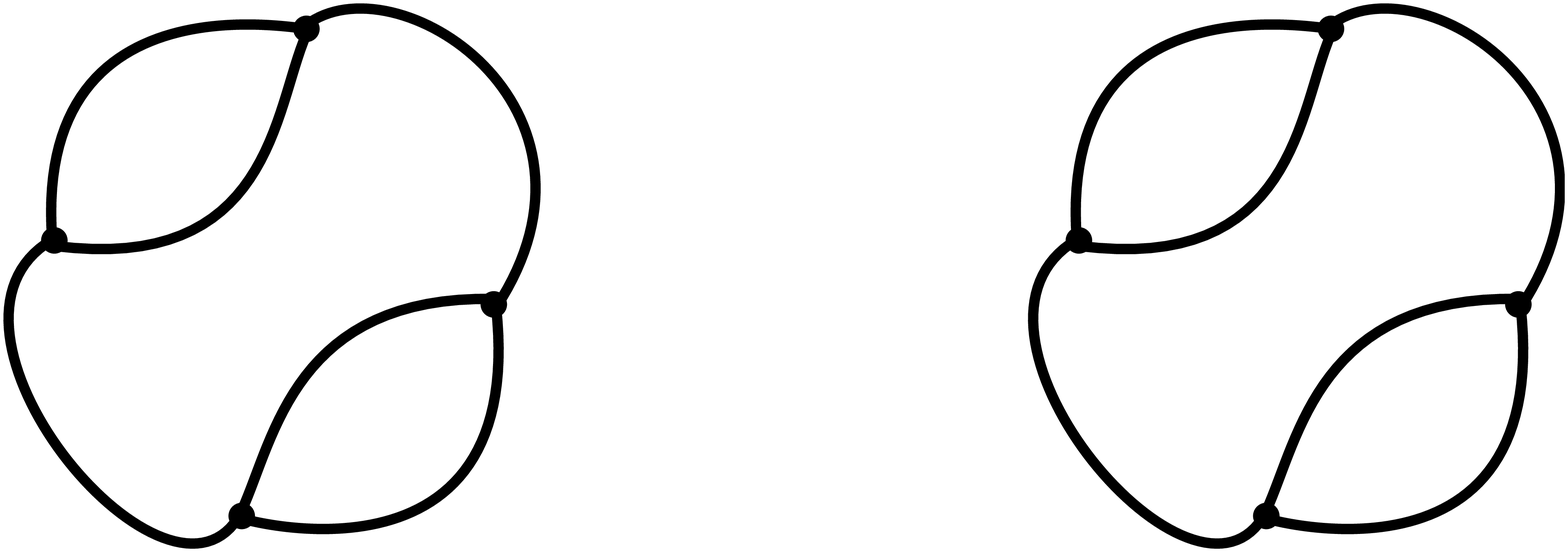}
        \small{
            \put(-60,5){$4$}
            \put(-41,58){$3$}
            \put(-41,37){$3$}
            \put(-10,35){$2$}
            \put(-32,19){$5$}
            \put(-21,5){$3$}}
            \put(-170,5){$5$}
            \put(-149,58){$3$}
            \put(-149,37){$1$}
            \put(-118,35){$3$}
            \put(-141,19){$4$}
            \put(-131,5){$8$}
        \caption{On the left, an example of a inadmissible quantum spin network. On the right, an example of an admissible quantum spin network.}
    \label{QSN}}
\end{figure}

 If $D$ is a quantum spin network in $S^2$ then the \textit{Kauffman bracket evaluation of $D$}, denoted $\left\langle D\right\rangle$ where $\left\langle \cdot \right\rangle:\mathcal{S}(S^2)\to\mathbb{Q}(A)$ is defined to be the evaluation of $D$ as an element in $\mathcal{S}(S^{2})$ after replacing any edge colored $n$ by $f^{(n)}$ and each vertex colored $(a,b,c)$ by the skein element $\tau_{a,b,c}$, as in Figure~\ref{replace}.
 If $D$ is inadmissible then we define $\left\langle D\right\rangle=0$.
 Often in this paper we will not distinguish between the QSN in $S^2$ and its evaluation as a linear skein of $S^2$.
 Finally, when we work with QSNs in other skein modules that are not necessarily $S^2$, one often needs to switch between the QSN and the corresponding skein element as illustrated in Figure~\ref{replace}.
 We will denote by $\left\langle D \right\rangle$ to skein element that corresponds to the quantum spin network $D$.
 
\begin{figure}[h]
    \centering
    {\includegraphics[scale=0.17]{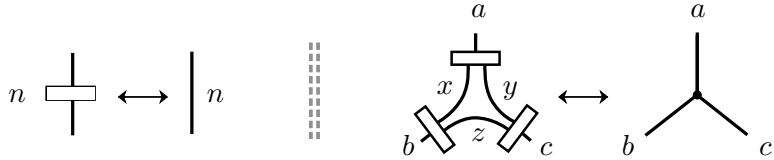}
        \put(-48,0){$b$}
        \put(-22,52){$a$}
        \put(-118,23){$x$}
        \put(-93,23){$y$}
        \put(-105,5){$z$}
        \put(4,0){$c$}
        \put(-131,0){$b$}
        \put(-105,52){$a$}
        \put(-79,0){$c$}
        \put(-280,20){$n$}
        \put(-205,20){$n$}
        \caption{The evaluation of a quantum spin network in the linear skein of $S^2$ is obtained by the above local replacement rules.}
    \label{replace}}
\end{figure}


\section{The Tail of Quantum Spin Networks}
\label{TQSN}

Let $D$ be a planar trivalent graph in $S^2$.
We label every edge in $D$ by $n$ or $2n$ where $n \in \mathbb{N}$ such that we obtain an admissible quantum spin network $D_n$.
This way we construct a sequence of quantum spin networks $\mathcal{D}=\{D_n\}_{n \in \mathbb{N}}$. In this paper we are concerned with the stability properties associated to the coefficients of such a sequence $\mathcal{D}$.  
We will show that if the sequence of quantum spin networks $\mathcal{D}$ satisfies a natural condition, then the first $n$ coefficients of $\left\langle D_n\right\rangle$ are identical to the first $n$ coefficients of $\left\langle D_{n+1}\right\rangle$ up to a common sign.
This stability gives rise to a $q$-series called the \textit{tail} of the quantum spin network $\mathcal{D}$.
The tail 
of such sequences were studied by the second author in~\cite{Hajij1}.
An investigation of QSNs follows organically from considering the tail of the colored Jones polynomial (see for example~\cite{Armond1}).
The relationship between the colored Jones polynomial and Rogers-Ramanujan identities is also studied in~\cite{CodyOliver}.
In this section we constrain ourselves to study the stability of admissible QSNs whose edges are all colored $2n$ and show that in this case the tail always exists.
 
Let $P_1(q)$ and $P_2(q)$ be non-zero power series in $\mathbb{Z}[q^{-1}][[q]]$.
For a positive integer $n$, we say that $P_1$ and $P_2$ are \textit{$n$-equivalent} and write $P_1(q)\doteq_n P_2(q)$, if their first $n$ coefficients agree up to a common sign.
For instance, $-q^{-4} + 15q^{-3} - 6 + 11 q \doteq_5 1 - 15q + 6q^4$.
When $P_1(q) \doteq_n P_2(q)$ for every integer $n \geq 0$, we will simply write $P_1(q) \doteq P_2(q)$.
We will denote the minimal degree of an element $f \in  \mathbb{Z}[q^{-1}][[q]] $ by $deg(f)$.
We now give the definition of the tail of sequence of elements in $\mathbb{Z}[q^{-1}][[q]]$.
\begin{definition}
\label{df}
	Let $\mathcal{P}=\{P_n(q)\}_{n \in \mathbb{N}}$ be a sequence of formal power series in $\mathbb{Z}[q^{-1}][[q]]$.
	A tail (if it exists) of the sequence $\mathcal{P}$ is a formal power series $T_{\mathcal{P}}(q) $ in $\mathbb{Z}[q^{-1}][[q]]$ with:
	\begin{equation*}
	\label{main definition}
	    T_{\mathcal{P}}(q)\doteq_{n}P_n(q), \text{ for all } n \in \mathbb{N}.
	\end{equation*}
\end{definition}

\begin{remark}
    One can immediately see from this definition that the sequence $\mathcal{P}=\{P_n(q)\}_{n \in \mathbb{N}}$ admits a tail if and if only if $P_{n}(q)\doteq_{n}P_{n+1}(q)$ for all $n$.
\end{remark}

\begin{remark}
    A tail of a sequence $\mathcal{P}$ described in Definition~\ref{df}, when it exists, is not unique.
    Namely, if $\mathcal{P}=\{P_n(q)\}_{n \in \mathbb{N}}$ is a sequence with a tail $T^\prime(\mathcal{P})$, then  $ q^{a}\cdot T^\prime(\mathcal{P})$ will also be a tail of $\mathcal{P}$ for any $a \in \mathbb{Z}$.
    Given any tail $T^\prime(\mathcal{P})$ we can always choose an $a$ such that $deg(T(\mathcal{P}))=0$, when $T(\mathcal{P})= q^{a} \cdot T^\prime(\mathcal{P})$.
    By convention we will refer to \textit{the} tail of the sequence $\mathcal{P}$ as \textit{the} power series $T(\mathcal{P})$ satisfying this condition, noting that $T(\mathcal{P})\in\mathbb{Z}[[q]]$.
    However, in calculating the tail of a sequence, it will be convenient to consider tails of $\mathcal{P}$ where $deg(T^\prime(\mathcal{P}))\neq 0$.
\end{remark}

The evaluation of a quantum spin network in $\mathcal{S}(S^2)$ yields an rational function in $\mathbb{Q}(A)$. Following Armond \cite{Armond1} we need to express a rational function as Laurant series in such a way that this Laurent series has a minimum degree. 
In other words every element in $\mathbb{Q}(A)$ is identified with a unique element in $\mathbb{Z}[A^{-1}][[A]]$. Hence the evaluation of any admissible quantum spin network in $\mathcal{S}(S^2)$ can be uniquely identified with a power series in $\mathbb{Z}[A^{-1}][[A]]$. We will use this identification to apply the previous equivalence relation, $\doteq_{n}$, to quantum spin networks. We will call this element in $\mathbb{Z}[A^{-1}][[A]]$  \textit{the power series evaluation of a quantum spin network}.  Moreover when working with the tail of such networks one often uses the variable $q$ instead of the variable $A$ (recall that $q=A^4$). In this case the power series evaluation of a quantum spin network is normalized, by multiplying by a $q^{a}$ for some power $a$, so that final power series is an element in $\mathbb{Z}[q^{-1}][[q]]$. We will assume this identification and normalization in what follows.

Now let $D$ be a trivalent graph and denote by $E(D)$ its edge set.
Let $\mathcal{F}=\{f_n:E(D)\longrightarrow \mathbb{N}\}_{n\in \mathbb{N}} $ be a sequence of label assignments on the edges of $D$ such that the resulting QSN is admissible for every $n \in \mathbb{N}$.
We will denote a quantum spin network $D$ labeled with $f_n$ by $D_{f_{n}}$ or simply by $D_n$ (when there is no confusion).

The tail of the sequence $\{D_n\}_{n \in \mathbb{N}}$, denoted by $T(D)$, is a series in $\mathbb{Z}[q^{-1}][[q]]$ that is $n$-equivalent to $D_n$ for each $n \in \mathbb{N}$.

The following two lemmas have straightforward proofs and will be useful when trying to compute tails.
\begin{lemma}
\label{easy}
    Let $P_1,P_2,Q_1,Q_2$ be non-zero power series in $\mathbb{Z}[q^{-1}][[q]]$ with $P_1(q)  \doteq_n P_2 (q) $ and $Q_1(q)  \doteq_n Q_2 (q) $, then
    \begin{enumerate}
        \item $P_1Q_1 \doteq_n P_2Q_2$.
        \item If $deg(Q_1) =deg(P_1) +a $ where $a > n$, then $P_1 \pm Q_1 \doteq_n P_1$.
        \item If $deg(P_1) =deg(P_2)$ and $deg(Q_1) =deg(Q_2)$, then $P_1+Q_1 \doteq_n  P_2+Q_2.$ 
    \end{enumerate}
\end{lemma}

Let $R$ be an element in $\mathbb{Q}(q)$ of the form $1/P$ where $P$ is a Laurent polynomial.
Then as we mentioned earlier we can write $R$ as an element in $\mathbb{Z}[q^{-1}][[q]]$.
Using this convention and part (1) of \ref{easy} we also obtain the following Lemma.
\begin{lemma}
\label{easy2}
    Let $P_1,Q_1,Q_2$ be non-zero power series in $\mathbb{Z}[q^{-1}][[q]]$ and let $P_2= 1/P \in  \mathbb{Q}(q)$ for $P \in \mathbb{Z}[q,q^{-1}] $.
    Furthermore, suppose that  $P_1(q)  \doteq_n P_2 (q) $ and $Q_1(q)  \doteq_n Q_2 (q) $, then  $P_1Q_1 \doteq_n P_2Q_2$.
\end{lemma}

\subsection{Adequate Quantum Spin Networks}

Let $D$ be a skein element in $\mathcal{S}(S^2)$ consisting of arcs and circles labeled by Jones-Wenzl idempotents of color $n$ or $2n$.
Let $\overline{D}$ denote the diagram obtained from $D$ by replacing each $n$-labeled arc with the idempotent $f^{(n)}$ by $n$ parallel arcs passing under it.
We say that the skein $D$ is adequate if in $\overline{D}$ each continuous arc passes at most once under any idempotent $f^{(n)}$.
Figures~\ref{smoothings1} and~\ref{adequate} show examples of adequate and non-adequate skein elements.
Each circle in the Figures bounds a disk, a result of our restriction of $f_n(e)\in\{n,2n\}$.
In Figure~\ref{smoothings1} each unlabeled arc represents $n$ parallel strands.
On the left, each strand passes multiple idempotents, but only passes under any individual idempotent once, while on the right, each arc passes under each idempotent twice.
A quantum spin network is adequate if its corresponding skein element is adequate, hence any adequate QSN must also be admissible.
The following theorem is due to~\cites{Armond1,Hajij2}.

\begin{figure}[h]
    \centering
	{\includegraphics[scale=0.15]{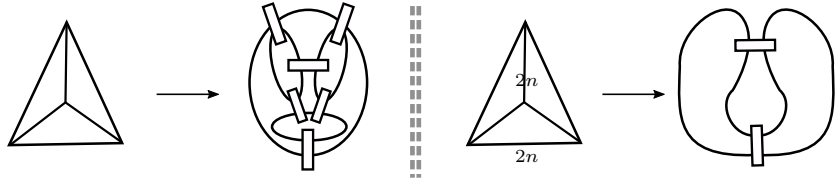}
		\put(-121,+6){\tiny{$2n$}}
		\put(-121,+35){\tiny{$2n$}}
		\caption{ All arcs in the skein elements are labeled by $n$.
		    Left: An adequate quantum spin network (all edges in the graph have label $2n$).
		    Right: A non-adequate quantum spin network (all non-labeled edges in the graph have label $n$).}
	\label{smoothings1}}
\end{figure}

\begin{theorem}
\label{mustafa_cody_theorem}
	Let $D$ be a trivalent graph, with a sequence of label assignment functions
	\begin{equation*}
        \mathcal{F}=\left\{f_n:E(D)\to \mathbb{N} \mid f_n(e) \text{ is equal to } n \text{ or } 2n  \mathrm{\ for\ all\ } e \in E(D) \right\}_{n\in \mathbb{N}}
	\end{equation*}
    on the edges of $D$ such that $D_{f_n}$ is an adequate quantum spin network for every $n$.
    Then the tail of $\left\{D_{f_n}\right\}_{n\in \mathbb{N}}$ exists.
\end{theorem}

\begin{remark}
    It is worth mentioning here that the tail of non-adequate QSNs may exist.
    For instance, the tail of the non-adequate QSN shown on the right of Figure~\ref{smoothings1} has been computed~\cite[Example 4.17]{Hajij2}.
\end{remark}

In this paper we will focus on sequences of quantum spin networks where the edges are all labeled $2n$.
The tail of such networks always exist, thanks to the following proposition.
\begin{proposition}
\label{existance}
	Let $D$ be a trivalent graph, and let $\left\{f_n :E(D) \to \mathbb{N}\right\}_{n\in \mathbb{N}}$ be a sequence of label assignments defined by $f_n(e)=2n$ for every $e \in E(D) $ and for every $n \in \mathbb{N}$.
	Then the tail of $\{D_{f_n}\}_{n\in \mathbb{N}}$ exists. 
\end{proposition}
\begin{proof}
    Let $F$ be an arbitrary face in $D$.
    Since edges are labeled by $2n$ everywhere then the skein element that corresponds to $D_{f_n}$ appears around $F$ as illustrated in Figure~\ref{adequate}.
    Every such face has an equivalent skein element as illustrated in Figure~\ref{adequate} on the right.
    If we replace each idempotent in $D$ by $2n$ parallel strands then we obtain $n$ parallel circles within this polygon such that each circle passes at most once under each former idempotent.
    This proves that $D_{f_n}$ is adequate for every $n$ and each $f_n(e)$ is linear.
    Hence by Theorem~\ref{mustafa_cody_theorem} the tail of $\left\{D_{f_n}\right\}$ exists.
    \begin{figure}[h]
    	\centering
	    {\includegraphics[scale=0.4]{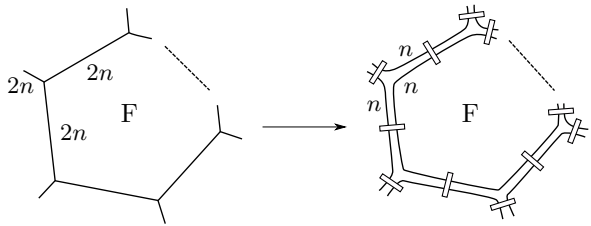}
	    	\put(-70,+50){\footnotesize{$n$}}
    		\put(-72,+64){\footnotesize{$n$}}
		    \put(-84,+43){\footnotesize{$n$}}
            \put(-48,+39){\small{F}}
    		\put(-190,+55){\footnotesize{$2n$}}
		    \put(-200,+32){\footnotesize{$2n$}}
		    \put(-220,+50){\footnotesize{$2n$}}
	    	\put(-177,+39){\small{F}}
    		\caption{A local picture of any face, F, of a quantum spin network whose edges are labeled by $2n$, on the left, and its corresponding skein    element on the right.}
	    \label{adequate}}
    \end{figure}
\end{proof}	


\section{Connections to the tail of the colored Jones polynomial}
\label{Application}

In this section we relate the tail of a quantum spin network to the tail of the colored Jones polynomial first discussed by Dasbach and Lin in~\cite{DL}. We show that the tail of trivalent graphs that are colored by $2n$ can be realized as the tail of an alternating link. 

The Jones polynomial knot invariant is given by a Laurent polynomial in the variable $q$ with integer coefficients.
The Jones polynomial generalizes to an invariant $J_{K,V}^{\mathfrak{g}}(q)\in \mathbb{Z}[q^{\pm 1}]$ of a zero-framed knot $K$ colored by a representation $V$ of a simple Lie algebra $\mathfrak{g}$, and normalized so that $J_{O,V}^{\mathfrak{g}}(q)=1$, where $O$ denotes the zero-framed unknot.
The invariant $J_{K,V}^{\mathfrak{g}}(q)$ is called the quantum invariant of the knot $K$ associated with the simple Lie algebra $\mathfrak{g}$ and the representation $V$.
The Jones polynomial corresponds to the second dimensional irreducible representation of $\mathfrak{sl}(2,\mathbb{C} )$ and the \textit{$n$-th colored Jones polynomial}, denoted by $J_{n,K}(q)$, is the quantum invariant associated with the $n+1$-dimensional irreducible representation of $\mathfrak{sl}(2,\mathbb{C} )$.

Dasbach and Lin observed in~\cite{DL} that, up to a common sign change, the first $n$ coefficients of $J_{n,L}(q)$ agree with the first $n$ coefficients of $J_{n+1,L}(q)$ for an alternating link $L$.
As seen in the following example:
\begin{example}
	The colored Jones polynomial of the knot $6_2$, up to multiplication with a suitable power $q^{\pm a_n}$ for some integer $a_n$, is given in the following table:
	\begin{center}
		\begin{small}
		\begingroup
            \setlength{\tabcolsep}{0.1em} 
    			\begin{tabular}{ |c|rllllllll|} 
			    	\hline
		    		\:$n=1$\:~&\:$1$&&&&&&&&\\
	    			$n=2$ & $1$& $-2q$& $+2q^{2}$& $-2q^{3}$& $+2q^{4}$& $-\:\;q^{5}$& $+\:\;q^{6}$&&  \\
    				$n=3$ & $1$& $-2q$&& $+4q^{3}$& $-5q^{4}$&& $+6q^{6}$& $-6q^{7}$& $+\cdots$\:~  \\ 
				    $n=4$ &  $1$& $-2q$& &$+2q^{3}$& $+\:\;q^{4}$& $-4q^{5}$& $-2q^{6}$& $+7q^{7}$& $+\cdots$\\ 
			    	$n=5$ & $1 $& $- 2 q$& &$ + 2 q^3 $& $-\:\;q^4$& $ + 2 q^5 $& $- 6 q^6$& $ + 2 q^7$& $ +\cdots$\\ 
		    		$n=6$ & $1$& $-2q$&& $+2q^{3}$& $-\:\;q^{4}$&&& $-2q^{7}$& $+\cdots$ \\ 
	    			$n=7$ & $1$& $-2q$&& $+2q^{3}$& $-\:\;q^{4}$&& $-2q^{6}$& $+4q^{7}$& $+\cdots$ \\ 
    				$n=8$ & $1$& $-2q$&& $+2q^{3}$& $-\:\;q^{4}$&& $-2q^{6}$& $+2q^{7}$& $+\cdots$\\
				    \hline 
			    \end{tabular}
            \endgroup
		\end{small}
	\end{center}
	hence the tail of the colored Jones polynomial of the knot $6_2$ is given by
	\begin{equation*}
		T_{6_2}(q)=1-2q+0q^2+2q^{3}-q^{4}+0q^{5}- 2q^{6}+2q^{7}+\cdots
	\end{equation*}
\end{example}
In~\cite{Armond1} Armond reduced the study of the tail of the colored Jones polynomial of an alternating link $L$ to a simpler sequence of skein elements in $\mathcal{S}(S^2)$ obtained from an alternating diagram of $L$. We recall Armond's result in detail here.


Let $D$ be a diagram of a link $L$ in $S^2$.
Any crossing of $D$ can be smoothed in two ways, either by the $A$-smoothing or by the $B$-smoothing illustrated in Figure~\ref{smoothings}.
By applying an $A$-smoothing or a $B$-smoothing to every crossing in $D$, one obtains a collection of circles called a \textit{Kauffman state} of the diagram $D$.
Let $c(D)$ be the crossing number of the diagram $D$, thus there are $2^{c(D)}$ Kauffman states.
Among these states, two particular states are important to us, namely, the all-$A$ smoothing (the state in which all crossings were replaced by the $A$-smoothing) denoted by $S_A(D)$, and the all-$B$ smoothing denoted by $S_B(D)$.

\begin{figure}[h]
	\centering
	{\includegraphics[scale=0.16]{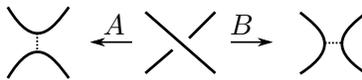}
		\put(-100,+17){$A$}
		\put(-52,+17){$B$}
		\caption{The $A$-smoothing and the $B$-smoothing of a crossing. }
	\label{smoothings}}
\end{figure}



For an alternating reduced diagram $D$ of a link $L$, a result of Kauffman~\cites{Kauffman1} states that the highest and lowest coefficients of the Kauffman bracket evaluation of such a diagram $D$ are equal to the highest and lowest coefficients of the $S_A(D)$ and $S_B(D)$ respectively.
An analogue of this result was proven for adequate links by Armond~\cites{Armond1}.
Starting with $S_B(D)$, consider the skein element obtained by decorating each circle in $S_B(D)$ with the $n^{th}$ Jones-Wenzl idempotent and replacing each dashed line in $S_B(D)$ with the $(2n)^{th}$ Jones-Wenzl idempotent.
Denoting this skein element  $S^{\,n}_B(D)$ (see Figure~\ref{graphs} for an example), the following result holds.
\begin{theorem}(Armond~\cite{Armond1})
	\label{cody thm}
	Let $L$ be a link in $S^3$ with reduced alternating knot diagram $D$.
	Then
	\begin{equation*}
		J_{n,L}(q)\doteq_{(n+1)}S_B^{\,n}(D).
	\end{equation*} 
\end{theorem}
By this theorem, the tail of the colored Jones polynomial is determined by the sequence $\left\{S_B^{\,n}(D)\right\}_{n\in \mathbb{N}}$.  
Now, for every $n$, the skein element $S_B^{\,n}(D)$ can be written as a trivalent graph in $\mathcal{S}( S^2)$. 
The element $S_B^{\,n}(D)$ can be written as quantum spin network by using the following simple identity :

\begin{eqnarray}
\label{skeine}
    \left\langle   \hspace{0.4cm}
    \begin{minipage}[h]{0.09\linewidth}
        \vspace{0pt}
        \scalebox{0.3}{\includegraphics{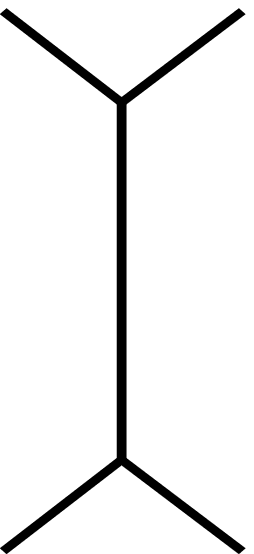}}
        \put(-29,+46){\footnotesize{$n$}}
        \put(-29,-1){\footnotesize{$n$}}
        \put(2,-1){\footnotesize{$n$}}
        \put(2,+46){\footnotesize{$n$}}
        \put(-5,+23){\footnotesize{$2n$}}
    \end{minipage}\hspace{-3mm}
    \right\rangle=
    \begin{minipage}[h]{0.09\linewidth}
        \vspace{0pt}
        \scalebox{0.14}{\includegraphics{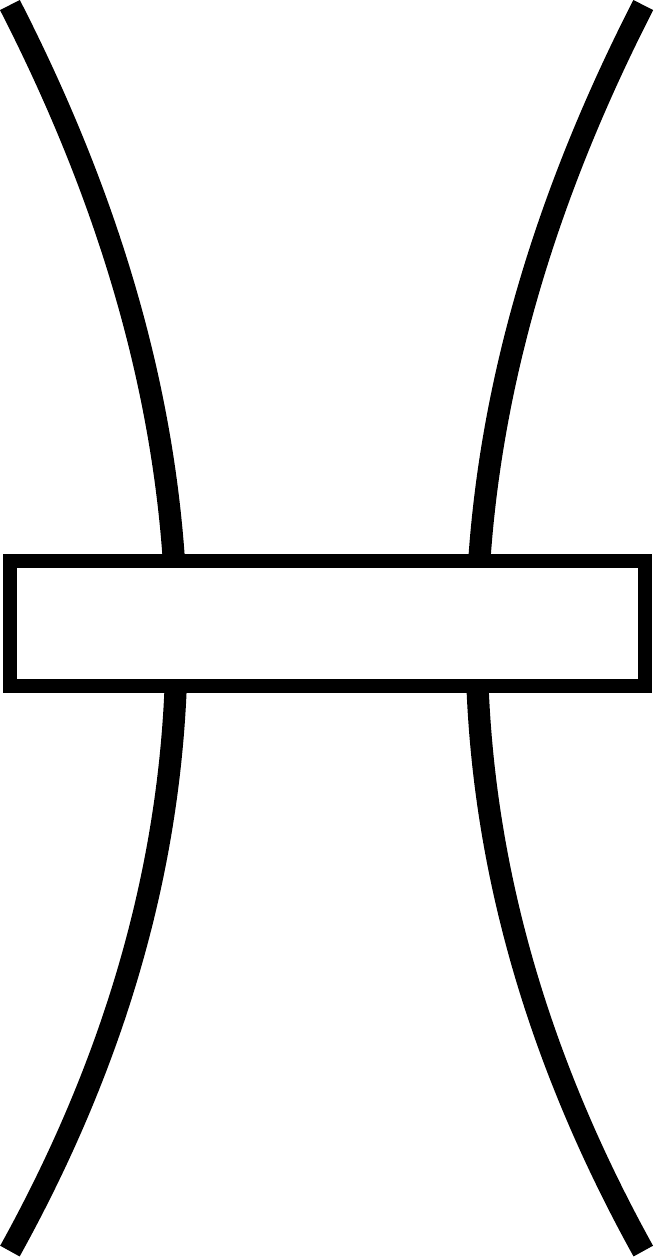}}
        \put(3,+50){\footnotesize{$n$}}
        \put(-35,+50){\footnotesize{$n$}}
    \end{minipage}
\end{eqnarray}
The following identity can also be useful to simplify the final QSN:
\begin{eqnarray}
\label{contraction}
    \left\langle \hspace{4pt}
    \begin{minipage}[h]{0.15\linewidth}
        \vspace{-0pt}
        \scalebox{0.35}{\includegraphics{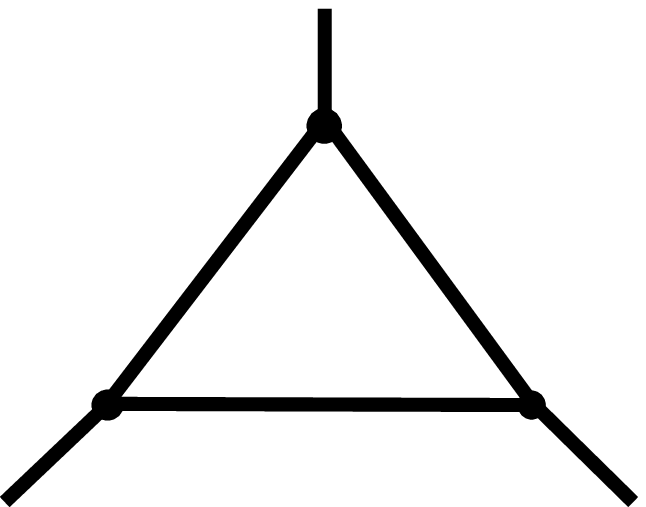}}
        \put(-69,-10){$2n$}
        \put(-1,-10){$2n$}
        \put(-28,50){$2n$}
        \put(-55,25){$n$}
        \put(-15,25){$n$}
        \put(-35,0){$n$}
    \end{minipage}
    \right\rangle \hspace{20pt} = \hspace{30pt}
    \begin{minipage}[h]{0.13\linewidth}
        \vspace{-0pt}
        \scalebox{0.25}{\includegraphics{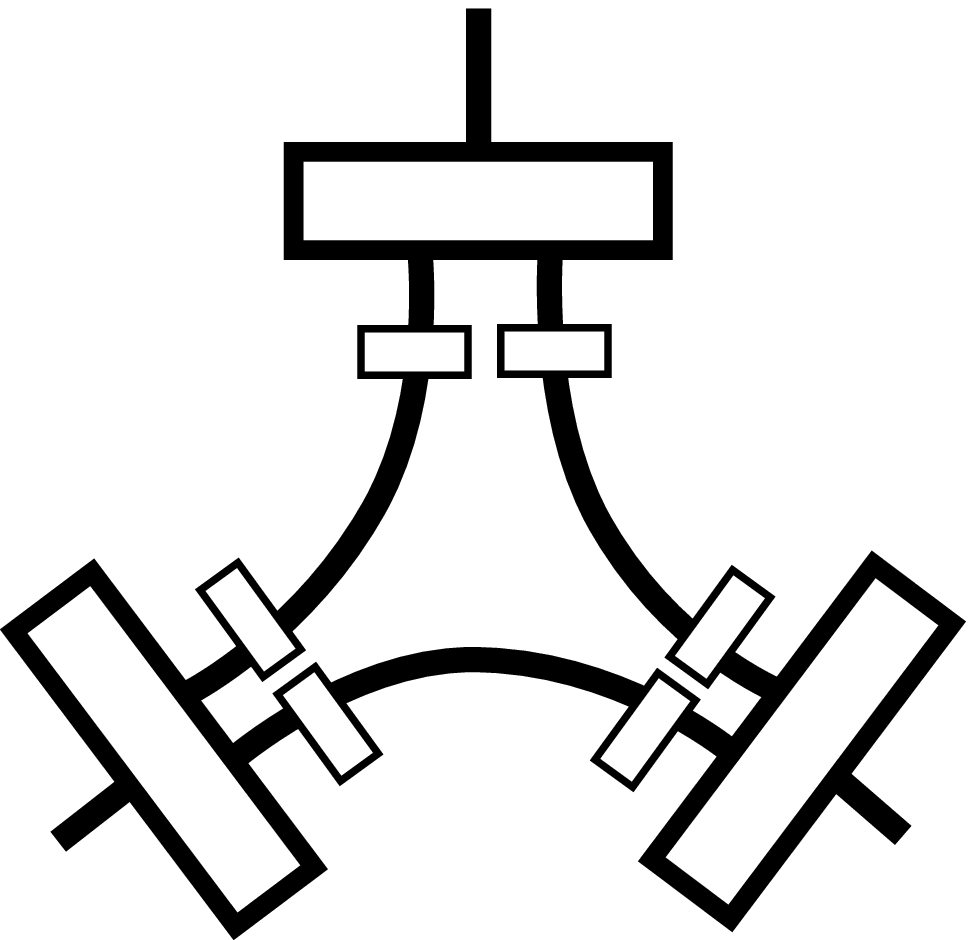}}
        \put(-70,-6){$2n$}
        \put(-10,-6){$2n$}
        \put(-30,62){$2n$}
    \end{minipage}
    \hspace{30pt}=\hspace{0.5cm}
    \left\langle \hspace{4pt}
    \begin{minipage}[h]{0.13\linewidth}
        \vspace{-0pt}
        \scalebox{0.25}{\includegraphics{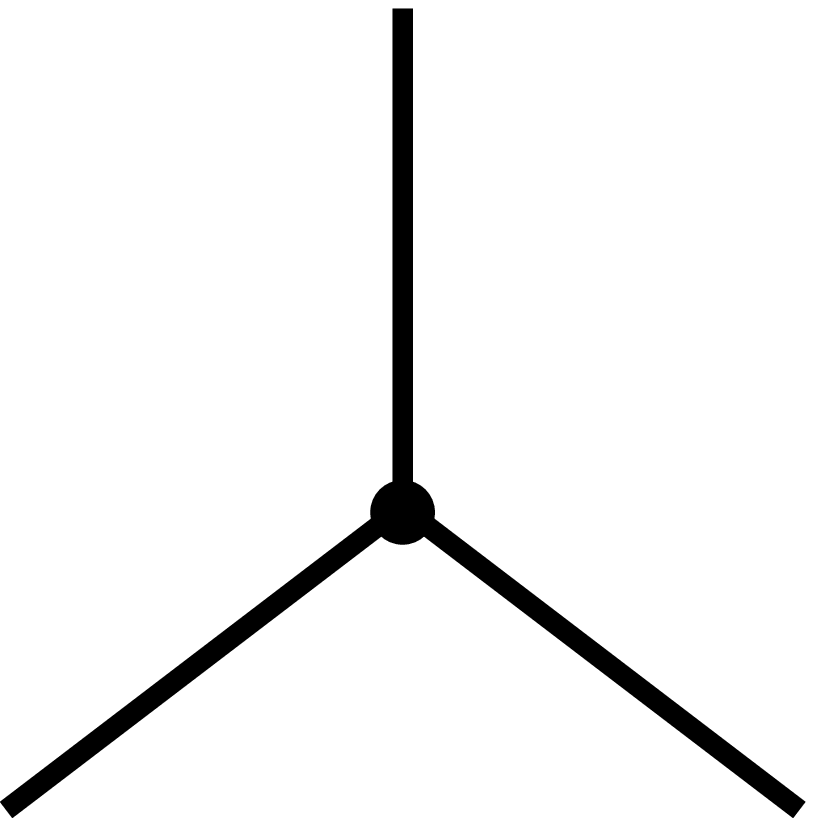}}
        \put(-60,-10){$2n$}
        \put(-1,-10){$2n$}
        \put(-26,50){$2n$}
    \end{minipage}
    \right\rangle
\end{eqnarray}
The left equality comes from equation~\ref{skeine} applied at each vertex, while the right equality was illustrated at the right of Figure~\ref{replace}.

Consider the trefoil $\mathcal{T}$ appearing in Figure~\ref{example123} on the left, the skein element  $S^{\,n}_B(\mathcal{T})$ appears on in the middle of the figure.
Using identity~\ref{skeine} we can obtain the trivalent graph equivalent to $S^{\,n}_B(\mathcal{T})$  shown on the right of Figure~\ref{example123}. Under identity~\ref{contraction} this has the same tail as the Theta graph which will be calculated in section~\ref{Computing}.

\begin{figure}[h]
    \centering
    {\includegraphics[scale=0.07]
        {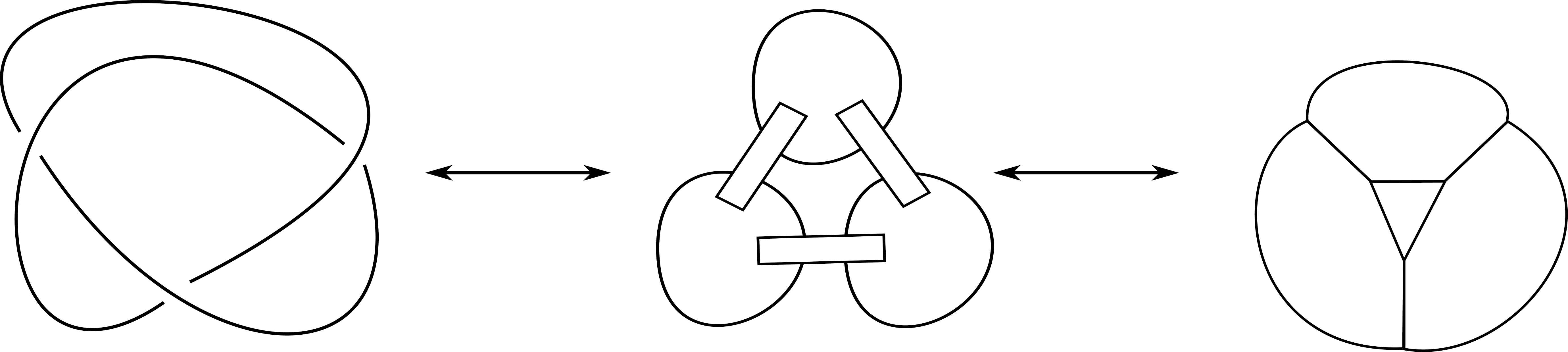}
        \small{
            \put(-14,23){\footnotesize{$2n$}}
            \put(-23,5){\footnotesize{$2n$}}
            \put(-40,23){\footnotesize{$2n$}}}
        \caption{From left to right: The trefoil $\mathcal{T}$, the skein element $S^{\,n}_B(\mathcal{T})$ and the corresponding trivalent graph of $S^{\,n}_B(\mathcal{T})$ obtained using identity~\ref{skeine}.
        All unlabeled edges in the trivalent graph are colored with $n$.
        All arcs in the skein element  $S^{\,n}_B(\mathcal{T})$  are labeled $n$.}
    \label{example123}}
\end{figure} 

In fact, given any trivalent graph $G$ corresponding to a quantum spin network with edges colored $2n$ we can construct an alternating link diagram $D$ such that the tail of the colored Jones polynomial of $D$ is equal to the tail of the sequence $\{G_{2n}\}_{n\in N}$ as we will explain in the following subsection.

\subsection{Going From Trivalent Graphs to Link Diagrams}
Given a trivalent graph $G$, the correspondence appearing in Figure~\ref{c2} can always be used to obtain an alternating link diagram $L(G)$.
In this paper we use the convention that all edges in trivalent graphs are replaced by negative twist regions as illustrated in the right of Figure~\ref{c2}.  

\begin{figure}[h]
\centering
    {\includegraphics[scale=0.12]{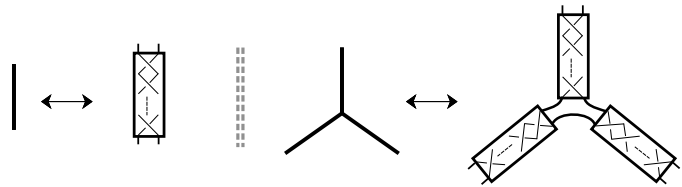}
        \caption{Obtaining a link diagram from a trivalent graph.}
    \label{c2}}
\end{figure} 

The tail of $J_{n,L(G)}$ can be seen to be equivalent to the tail of the trivalent graph sequence $\{G_{2n}\}_{n\in N}$ by observing that each skein element $S^{\,n}_B(L(G))$ corresponds precisely to the graph $G_{2n}$ under identities~\ref{skeine} and~\ref{contraction}.
It then follows from Theorem~\ref{cody thm} that the tail of the quantum spin network sequence $\{G_{2n}\}$ is equivalent to the tail of $J_{n,L(G)}$.
The correspondence given in Figure~\ref{c2} thus allows for the tail of any family of link diagrams corresponding to $L(G)$ for any trivalent graph $G$ to be computed.

\subsection{Reduced Graphs and the tail of the Colored Jones Polynomial}

The sequence $\{S_B^{\,n}(D)\}_{n\in \mathbb{N}}$ depends on a simple planar graph obtained from the knot diagram $D$.
We review this fact here.
To each Kauffman state $S(D)$ of a link diagram $D$, one can associate a graph $\mathbb{G}_{S(D)}$ obtained by replacing each circle of $S(D)$ by a vertex and each dashed line by an edge. See Figure \ref{graphs}.
In particular the \textit{$B$-graph}, denoted by $\mathbb{G}_B(D)$, is the graph obtained from the all $B$-state in this manner.
\textit{The reduced $B-$graph}, denoted by $\mathbb{G}^{\prime}_B(D)$, is obtained from the $B$-graph by keeping the same set of vertices of $B(D)$ and replacing parallel edges by a single edge.
See Figure~\ref{graphs} for an example. 

\begin{figure}[h]
\centering
    {\includegraphics[scale=0.07]{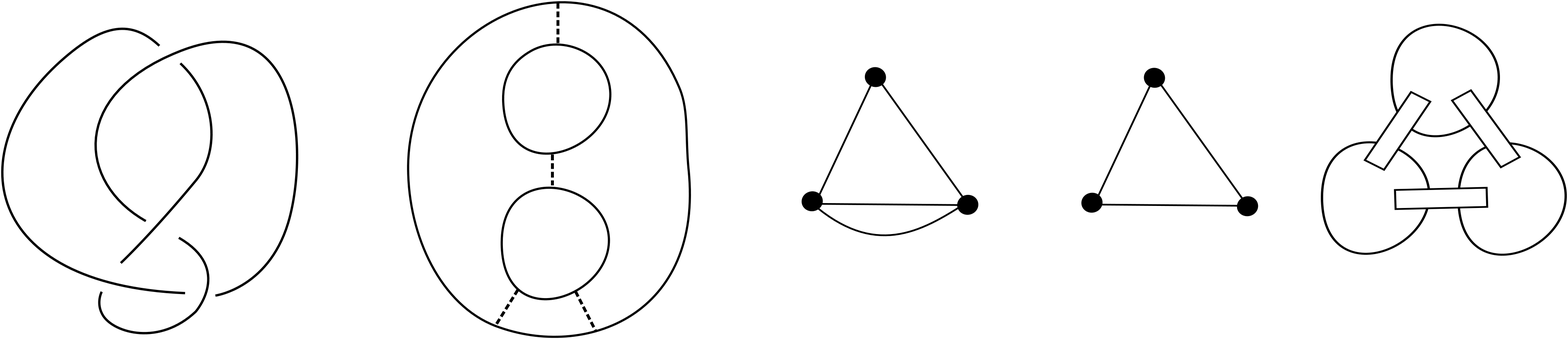}
        \caption{From left to right: A knot diagram $D$, its B-state, the B-graph $\mathbb{G}_B(D)$ of $D$, the reduced $B$-graph $\mathbb{G}^{\prime}_B(D)$ of $D$ and the skein element $S^{\,n}_B(D)$.  }
    \label{graphs}}
\end{figure} 
Using this definition of reduced $B$-graph, Theorem~\ref{cody thm} implies that the tail of colored Jones polynomial depends only on the reduced $B$-graph of $D$.
Thus  we will define the \textit{tail of a reduced graph $G$} to be the tail of the colored Jones polynomial of a link $D$ whose reduced $B$-graph is $G$. 

The reader is reminded that we are working with two kinds of graphs now: the trivalent graphs, denoted by $G$, and the $B$-graphs, denoted by $\mathbb{G}$.
Using Theorem~\ref{cody thm} we can meaningfully talk about the tail of colored Jones polynomial and its reduced graph.
Furthermore, since the reduced graph corresponds to skein element which in turn corresponds to a trivalent graph we may also refer to the tail of the colored Jones polynomial via its corresponding trivalent graph.
This correspondence between various graphs is illustrated in Figure~\ref{Correspondence}.

\begin{figure}[h]
\centering
    {\includegraphics[scale=0.14]{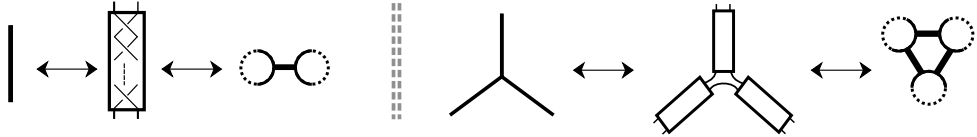}
        \caption{Correspondence between trivalent graphs (left), link diagrams (middle) and reduced graphs (right).}
    \label{Correspondence}}
\end{figure} 


\section{The Product Structure on Tails}\label{Prod}
\label{product}

Armond and Dasbach proved in~\cites{CodyOliver} that the tail of the edge connect sum of two graphs, $\mathbb{G}_i$ is equal to the multiplication of the tails of these two graphs.
We recall this product structure result before giving the natural analogue for it on the tail of quantum spin networks with edges colored $2n$.
The multiplication of two graphs is defined in Figure~\ref{product graphs}. 

\begin{figure}[h]
\centering
    {\includegraphics[scale=0.15]{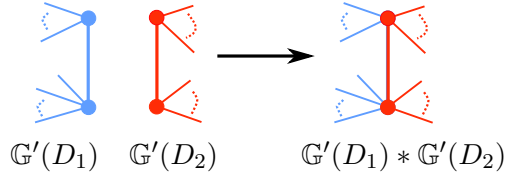}
        \put(-115,-15){$\mathbb{G}^{\prime}(D_2)$}
        \put(-160,-15){$\mathbb{G}^{\prime}(D_1)$}
        \put(-50,-15){$\mathbb{G}^{\prime}(D_1)*\mathbb{G}^{\prime}(D_2)$}
        \caption{Two reduced graphs and their product.}
    \label{product graphs}}
\end{figure} 

\begin{theorem}
\label{productcody}
    Let $D_1$ and $D_2$ be two reduced link diagrams.
    Then $T_{\mathbb{G}^{\prime}(D_1)}T_{\mathbb{G}^{\prime}(D_2)}=T_{\mathbb{G}^{\prime}(D_1)*\mathbb{G}^{\prime}(D_2)}$.
\end{theorem}

Now for our consideration, let $G_1$ and $G_2$ be trivalent graphs in $\mathcal{S}(S^2)$.
Suppose that each of $G_1$ and $G_2$ contains the trivalent graph $\tau_{2n,2n,2n}$ as in Figure~\ref{graph123}.

\begin{figure}[h]
\centering
    {\includegraphics[scale=0.27]{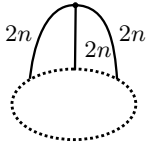}
        \footnotesize{
            \put(-50,38){$2n$}
            \put(-7,38){$2n$}
            \put(-20,32){$2n$}}
        \caption{The graph $G$ with a trivalent graph $\tau_{2n,2n,2n}$}
    \label{graph123}}
\end{figure} 

Define the map
$\; <,>:\mathcal{S}(S^2)\times\mathcal{S}(S^2)\longrightarrow\mathcal{S}(S^2)$
by the wiring illustrated in Figure~\ref{product_wiring}.
The theta graph provides a natural identity for this multiplication.

\begin{figure}[h]
\centering
    {\includegraphics[scale=0.09]{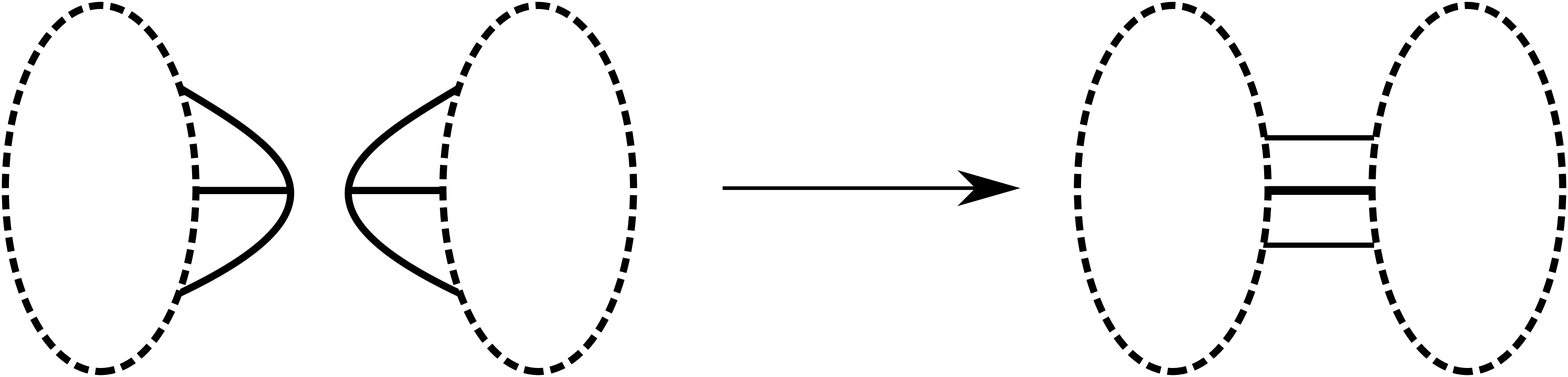}
        \footnotesize{
            \put(-135,-10){$G_2$}
            \put(-180,-10){$G_1$}
            \put(-50,-10){$<G_1,G_2>$}}
        \caption{The product $<G_1,G_2>$}
    \label{product_wiring}}
\end{figure}
The set of tails of the trivalent graphs with edges colored $2n$ behaves well under the product of Figure~\ref{product_wiring} as described in the following theorem.
\begin{theorem}
\label{codyy}
    Let $G_1$ and $G_2$ be as defined above.
    Then
    \begin{equation*}
        T_{<G_1,G_2>}=\frac{1}{(q^2;q)_{\infty}} T_{G_1}T_{G_2}
    \end{equation*}
\end{theorem}
\begin{proof}
    The skein space $Y_{2n,2n,2n}$ is one dimensional and generated by the skein element $\tau_{2n,2n,2n}$.
    Hence we can write 
    \begin{equation*}
        G_i=R_i(q)\Theta(2n,2n,2n),
    \end{equation*}
    for some $R_i(q) \in \mathbb{Q}(q)$ for $i=1,2$.
    Moreover, by applying the same fact to the diagram $<G_1,G_2>$, one graph at a time, we have:
    \begin{equation*}
        <G_1,G_2>=R_1(q)R_2(q)\Theta(2n,2n,2n).
    \end{equation*}
    By our assumption that $G_1$ and $G_2$ are trivalent graphs with edges colored $2n$, proposition~\ref{existance} ensures that we have
    \begin{equation*}
        T_{G_i}\doteq_n R_i(q)\Theta(2n,2n,2n)
    \end{equation*}
    for $i=1,2$.
    Now, $T_{\Theta(2n,2n,2n)}=\frac{(q;q)_\infty}{1-q}$ as will be shown in Proposition~\ref{lemmaaa}.
    Thus,
    \begin{eqnarray*}
        <G_1,G_2>&\doteq_n&T_{G_1}\frac{T_{G_2}}{\Theta(2n,2n,2n)}\;\;\doteq_n\;\frac{1}{(q^2;q)_n}T_{G_1}T_{G_2}.
    \end{eqnarray*}
    The result follows.
\end{proof}

Returning to the case where every edge is labeled $2n$.
Observe that the multiplication defined by the wiring in Figure~\ref{product_wiring} does not depend on the choice of the trivalent vertex.
We now have the following immediate result.
\begin{corollary}
\label{identity}
    Let $G$ be any $2n$-colored trivalent graph in $\mathcal{S}(S^2)$
    , then $T_{<\Theta_{2n},G>}= T_{G}$. 
\end{corollary}
By the virtue of Proposition~\ref{existance} we can meaningfully speak about the set of tails of trivalent graphs with edges colored $2n$.
Denote this set by $\mathcal{G}$.
Let $G_1$ and $G_2$ be two elements in $\mathcal{G}$.
Define on the set $\mathcal{G}$ the product $*$ by $$T(G_1)*T(G_2)=(q^2;q)_{\infty}T _{<G_1,G_2>}.$$

In other words the multiplication of the of tails $T(G_1)$ and $T(G_2)$ is equal, up to a factor, to the tail of graph multiplications $<G_1,G_2>$.


\section{Computing the tail of the Theta and Tetrahedron graphs}
\label{Computing}

In this section we investigate the tail of the theta and tetrahedron graphs with edges colored $2n$.
These graphs can be used to compute the tail of infinite families of other graphs using both the product, $<\cdot,\cdot>$, and other techniques as we will demonstrate.
We start with the following useful Lemma.
\begin{lemma}
\label{good}
    Let $n$ be a positive integer and let $F(q,n)$ be a rational function of the form:
    \begin{equation}
    \label{haha}
        F(q,n)=\sum_{i=0}^n  P(q,n,i)
    \end{equation}
    where $P(q,n,i)$ is an element in $\mathbb{Q}(q)$ of the form $ P(q,n,i)=\frac{(q;q)_n}{(q;q)_{n-i}} Q(q,n,i)$ for some $Q(q,n,i) \in \mathbb{Q}(q)$.
    Suppose further that  
    $deg(P(q,n,i))+i < deg(P(q,n,i+1)) $ for all positive integers $n,i$ with $i\leq n$.
    Then $F(q,n)\doteq_n \sum_{i=0}^n Q(q,n,i).$
\end{lemma}
\begin{proof}
    Beginning with the relation 
    \begin{equation}
    \label{division}
        \frac{(q;q)_n}{(q;q)_{n-i}}=1-q^{n-i+1}+O(n-i+2),
    \end{equation}
    we have that $deg(P(q,n,i))=deg(Q(q,n,i))$ for all positive integers $n,i$.
    To simplify notation, we will denote $deg(P(q,n,i))$ by $d_{n,i}$.

    Now, for all $i\geq 0,$ it follows from the assumptions and equation~\ref{division} that 
    \begin{equation}
    \label{claim1}
        P(q,n,i)+P(q,n,i+1)\doteq_n P(q,n,i)+Q(q,n,i+1) 
    \end{equation}
    as
    \begin{small}
        \begin{eqnarray*}
            P(q,n,i)+P(q,n,i+1)&=&P(q,n,i)+ \frac{(q;q)_n}{(q;q)_{n-i-1}} Q(q,n,i+1)\\
            &=&P(q,n,i)+ (1-q^{n-i}+O(n-i+1))Q(q,n,i+1)\\
            &=& P(q,n,i)+ Q(q,n,i+1)-q^{n-i} Q(q,n,i+1) +O(n-i+1 + d_{n,i+1} )
        \end{eqnarray*}
    \end{small}
    and thus that $deg (P(q,n,i)+P(q,n,i+1))=deg (P(q,n,i)+Q(q,n,i+1))$.
    Moreover, since $d_{n,i+1}>d_{n,i}+i$ we have: 
    $$deg (-q^{n-i} Q(q,n,i+1))=n-i+d_{n,i+1}>n-i+d_{n,i}+i=d_{n,i}+n.$$ 

    Thus, the first $n$ coefficients of the terms $-q^{n-i} Q(q,n,i+1) +O(n-i+1 + d_{n,i+1} )$ do not contribute to the first $n$ coefficients of $P(q,n,i)+P(q,n,i+1)$ hence equation~\ref{claim1} holds.
    Now by applying~\ref{claim1} inductively to equation~\ref{haha}, we obtain:

    \begin{equation*}
        \sum_{i=0}^n  P(q,n,i) \doteq_n P(q,n,0)+ \sum_{i=1}^{n}  Q(q,n,i) 
    \end{equation*}
    Since $P(q,n,0)=Q(q,n,0)$ the results follows.
\end{proof}
In order to compute the tail of the theta and tetrahedron graphs we now recall a few identities from the skein theory associated to the Kauffman bracket. 
The exact formula of the tetrahedron and theta coefficients can be found in~\cites{MasVog} and we shall not repeat them in full here.
Using the following identity from~\cite{EH}:
\begin{equation*}
    \prod\limits_{i=0}^{j}[n-i]_q=q^{(2 + 3 j + j^2 - 2 n - 2 j n)/4} (1 - q)^{-1 - j}\frac{(q;q)_n}{(q;q)_{n-j-1}}
\end{equation*}
and the formula of the tetrahedron and theta graphs from~\cite{MasVog}, we obtain, after simplification, the following two identities:

\begin{equation}
\label{theta}
    \Theta(2n,2n,2n)= (-1)^{n} q^{-3n/2}\frac{(q; q)^3_n(q;q)_{3n+1}}{(1-q) (q;q)_{2 n}^3}.
\end{equation}
and 
\begin{equation}
\label{tet}
    Tet\left[ 
        \begin{array}{ccc}
            2n & 2n & 2n \\ 
            2n & 2n & 2n%
        \end{array}%
    \right]=\frac{q^{-2n}(q;q)^{12}_n}{(1-q)(q;q)^6_{2n}}\sum\limits_{i=0}^{n}\frac{(-1)^{i}q^{(i +3i^2)/2}(q;q)_{4n-i}}{(q;q)_{n-i}^4(q;q)_{i}^3}.
\end{equation}
We now compute the tail of the theta and tetrahedron graphs.
For simplicity of the notation we will denote $\Theta(2n,2n,2n)$ by $\Theta_{2n}$ and 
$Tet\left[ 
    \begin{array}{ccc}
        2n & 2n & 2n \\ 
        2n & 2n & 2n%
    \end{array}%
\right]$ by $H_{2n}$.

\begin{remark}
    Proposition~\ref{lemmaaa} part $(1)$ computes the tail of the theta graph which, following Theorem~\ref{cody thm} and the discussion of Figure~\ref{example123}, is equivalent to computing the tail of the trefoil.
    This provides an alternative method to earlier computations~\cites{Armond1,Hajij2}.
\end{remark}

\begin{proposition}
    \label{lemmaaa}
    The tails of the theta and tetrahedron graphs are given by:
    \begin{enumerate}
        \item $T(\Theta_{2n})= \frac{(q;q)_{\infty}}{1-q} $
        \item $T(H_{2n})=\frac{(q;q)^3_\infty}{(1-q)}\sum\limits_{i=0}^{\infty}\frac{(-1)^{i} q^{(i +3i^2)/2}}{(q;q)^3_i}$
    \end{enumerate} 
\end{proposition}
\begin{proof}
    \begin{enumerate}
    \item First observe that
        \begin{equation}
        \label{0}
            \frac{(q;q)_{n}}{(q;q)_{2n}}
            =\frac{\displaystyle\prod_{k=0}^{n-1}(1-q^{k+1})}{\displaystyle\prod_{k=0}^{2n-1}(1-q^{k+1})} =\frac{1}{\displaystyle\prod_{k=n}^{2n-1}(1-q^{k+1})}
            =\displaystyle\prod_{k=0}^{n-1}\frac{1}{(1-q^{n+k+1})}
            \doteq_n1.
        \end{equation}
        Moreover,
        \begin{eqnarray*}
        \label{1}
            (q;q)_{3n+1}  \doteq_n (q;q)_{\infty}.
        \end{eqnarray*}
        Then the result follows directly from equation~\ref{theta}.

    \item Starting from~\ref{tet}, and applying the equivalence~\ref{0} using Lemma~\ref{easy} implies that:
        $$\frac{q^{-2n}(q;q)^{12}_n}{(1-q)(q;q)^6_{2n}}\sum\limits_{i=0}^{n} 
            \frac{(-1)^{i}q^{(i +3i^2)/2}(q;q)_{4n-i}}{(q;q)_{n-i}^4(q;q)_{i}^3}
        \doteq_n (q;q)^{6}_n\sum\limits_{i=0}^{n} 
            \frac{(-1)^{i}q^{(i +3i^2)/2}(q;q)_{4n-i}}{(1-q)(q;q)_{n-i}^4(q;q)_{i}^3}.$$ 
        Now consider:
        \begin{equation*}
            F(n,q)=(q;q)^{2}_n\sum\limits_{i=0}^{n}P(q,n,i),
        \end{equation*}
        where 
        \begin{equation*}
            P(q,n,i)=\frac{(-1)^{i}q^{(i +3i^2)/2}(q;q)_{4n-i}(q;q)_{n}^4 }{(1-q)(q;q)_{n-i}^4(q;q)_{i}^3}.
        \end{equation*}
        Observe that $deg(P(q,n,i))=(i +3i^2)/2$, and hence $deg(P(q,n,i))+i> deg(P(q,n,i+1))$.
        Then Lemma~\ref{good} implies: 
        \begin{equation*}
            F(n,q)\doteq_n (q;q)^{2}_n\sum\limits_{i=0}^{n} 
                \frac{(-1)^{i}q^{(i +3i^2)/2}(q;q)_{4n-i}}{(1-q)(q;q)_{i}^3}
            \doteq_n \frac{(q;q)^{3}_n}{(1-q)}\sum\limits_{i=0}^{n} 
                \frac{(-1)^{i}q^{(i +3i^2)/2}}{(q;q)_{i}^3}.
        \end{equation*}
        The result thus follows.
    \end{enumerate}
\end{proof}
We usually work with the normalized colored Jones polynomial, hence it is more natural to also work with the normalized tail.
The normalizing is done by dividing the tail by $\Delta_{n}$.
To this end, we compute the the tail of $\Delta_{n}$. 
\begin{align*}
    \Delta_{n} \doteq_n [n+1]_q=_{~}& \frac{q^{(n+1)/2}-q^{-(n+1)/2}}{q^{1/2}-q^{-1/2}}\\
        =_{~}&\frac{-1}{q^{{-1/2}}}\times \frac{q^{(n+1)/2}-q^{-(n+1)/2}}{1-q}\\
        \doteq_n& (q^{(n+1)/2}-q^{-(n+1)/2} ) \sum_{i=0}^{\infty}q^i\doteq_n \sum_{i=0}^{\infty}q^i=\frac{1}{1-q}
\end{align*}

This can be used to simplify the fraction $\frac{1}{1-q}$ in the formulas obtained in Proposition~\ref{lemmaaa}.
The product of Theorem~\ref{codyy} and these results provide the building blocks to compute the tails of useful formulae and infinite families of alternating links as well.
We illustrate an example of such computations below.

Given any $2n$-colored trivalent graph whose tail we know, we can either add or contract triangular faces.
From~\cites{MasVog} we have the following equality:
\begin{eqnarray}
\label{firsty}
    \left\langle\hspace{3mm}
    \begin{minipage}[h]{0.21\linewidth}
        \vspace{-0pt}
        \scalebox{0.35}{\includegraphics{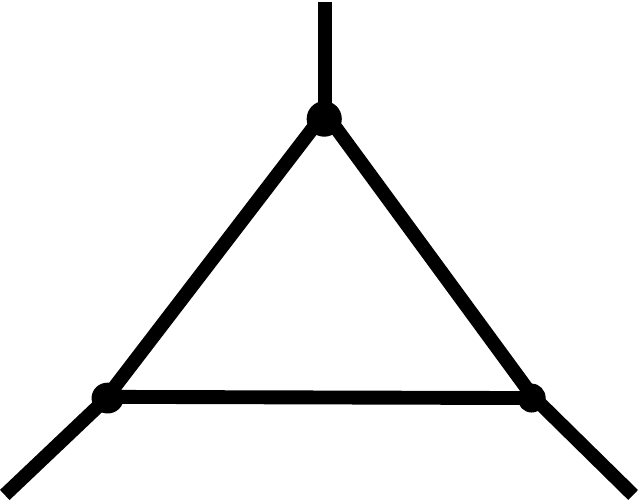}}
        \put(-71,-10){$2n$}
        \put(-3,-10){$2n$}
        \put(-28,50){$2n$}
        \put(-59,25){$2n$}
        \put(-16,25){$2n$}
        \put(-38,0){$2n$}
    \end{minipage}
    \hspace{-8mm}\right\rangle&=&\sigma(n)\left\langle\hspace{3mm}
    \begin{minipage}[h]{0.16\linewidth}
        \vspace{-0pt}
        \scalebox{0.25}{\includegraphics{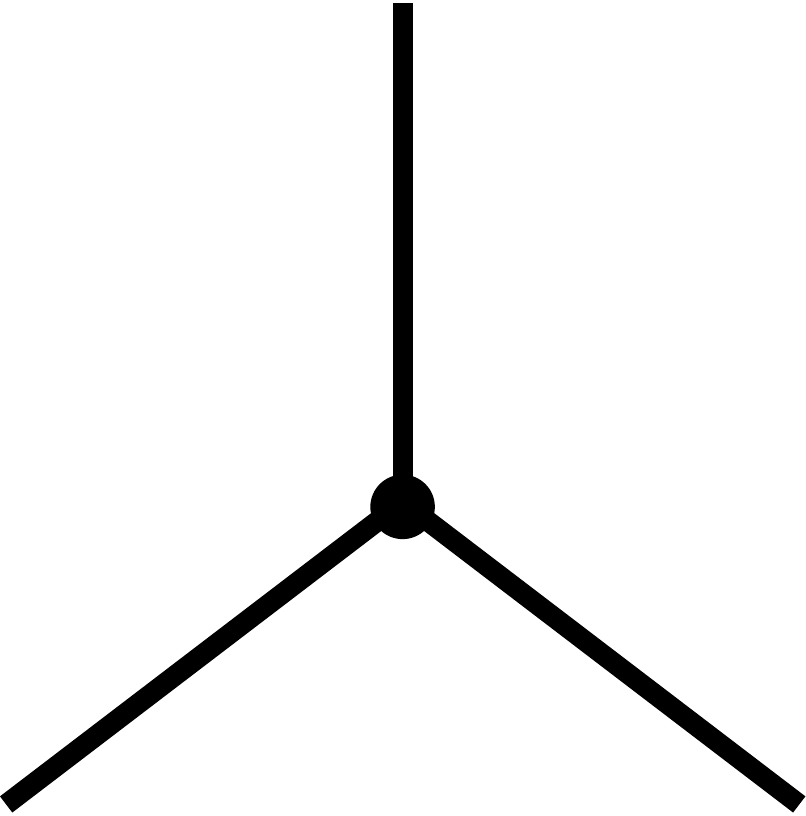}}
        \put(-60,-10){$2n$}
        \put(-1,-10){$2n$}
        \put(-26,50){$2n$}
        \put(-42,23){$Y$}
    \end{minipage}
    \hspace{-3mm}\right\rangle, \qquad\mathrm{ where }\quad\sigma(n)=\frac{H_{2n}}{\Theta_{2n}}.
\end{eqnarray}

\begin{example}
    Consider the multiplication of the two graphs given on the left hand side of Figure~\ref{graph}.
    We know from Lemma~\ref{lemmaaa} that the tail of $H_{2n}$ is given by $T_{H_{2n}}(q)$.
    Hence the tail of the graph given on the right hand side of Figure~\ref{graph}, which we will denote $\Gamma_{2n}$, is $$T_{<H_{2n},H_{2n}>}=\frac{1}{(q^2;q)_{\infty}} T_{H_{2n}}(q)^2$$.

    \begin{figure}[h]
    \centering
        {\includegraphics[scale=0.12]{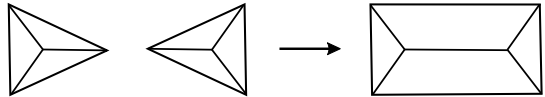}
            \caption{An example of the product $<,>$ on two trivalent graphs, creating $\Gamma_n$.}
        \label{graph}}
    \end{figure} 



    Alternatively, we could have made use of the identity~\ref{firsty} and applied it
    twice to contract the graph $\Gamma_{2n}$.
    The evaluation of the $\Gamma_{2n}$ graph is then given by:
    $$\Gamma_{2n}=\left(\frac{H_{2n}}{\Theta_{2n}}\right)^2\Theta_{2n}=\frac{H_{2n}^2}{\Theta_{2n}}.$$
    Then, using Proposition~\ref{lemmaaa} we could again obtain the tail of the graph $\Gamma_{2n}$.
    Using the correspondence illustrated in Figure~\ref{Correspondence}, these techniques give the tail of the infinite family of alternating links depicted in Figure~\ref{example2}.

    \begin{figure}[h]
    \centering
        {\includegraphics[scale=0.15]{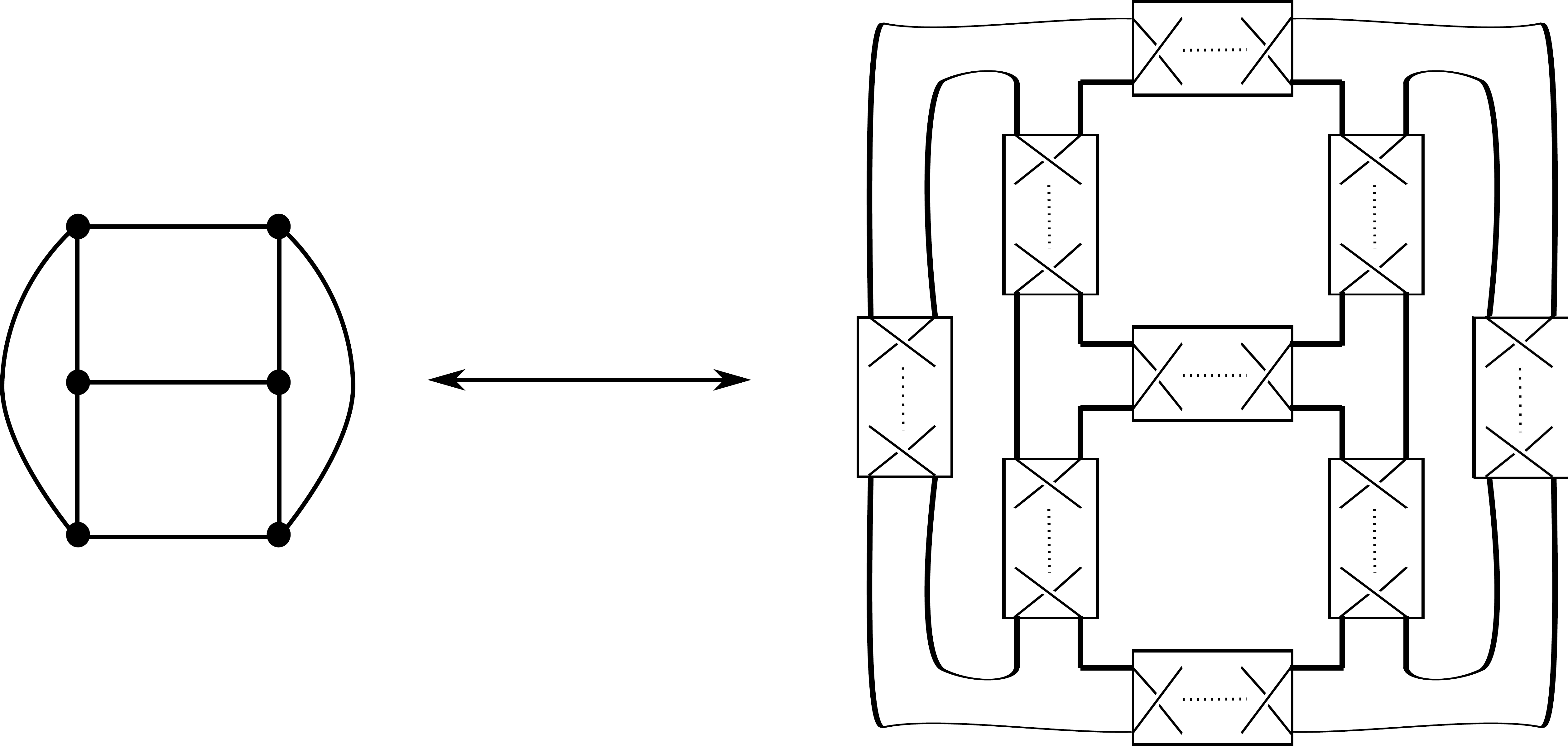}
            \caption{The tail of the graph $\Gamma_{n}$ appearing on the left is equivalent to the tail of the colored Jones polynomial of the family appearing on the right.}
        \label{example2}}
    \end{figure} 
\end{example}

\end{document}